\documentclass[12pt, showkeys]{amsart}%
\usepackage{color}
\usepackage{graphicx}
\usepackage{subfigure}
\usepackage{amssymb}
\usepackage{amsmath,amsxtra}
\usepackage{multirow}
\usepackage{enumerate}
\usepackage{epstopdf}
\usepackage{cite,stmaryrd,txfonts}
\usepackage{tikz}
\usepackage{pgf}
\usetikzlibrary{snakes}
\usetikzlibrary{angles,calc,patterns}
\usetikzlibrary{shapes,arrows,automata}
\usepackage{bbm}
\usepackage{dsfont}
\usepackage{accents}
 \usepackage{float}

\usepackage{epic}
\usepackage{empheq}
\usepackage{cases}
\usepackage{diagbox}
\def\dx{\,{\rm dx}}

\newtheorem{theorem}{Theorem}[section]
\newtheorem{remark}[theorem]{Remark}

\newtheorem{lemma}[theorem]{Lemma}

\newcounter{mnote}
\let\oldmarginpar\marginpar
\renewcommand\marginpar[1]{\-\oldmarginpar[\raggedleft\footnotesize #1]
  {\raggedright\footnotesize #1}}

\usepackage{footnote}
\usepackage{booktabs}

\usepackage{rotating}

\numberwithin{equation}{section}

\usepackage{cite}

\usepackage[colorlinks,linkcolor=black,
anchorcolor=black,
citecolor=black]{hyperref}
\usepackage[shortlabels]{enumitem}
\setlist[enumerate]{nosep}
\usepackage{color, soul}
\soulregister\cite7
\usepackage{setspace}

\def\rot{{\rm rot}}
\def\curl{{\rm curl}}

\def\dv{{\rm div}}

\def\prodg{\prod_{K\in\mathcal{G}}}
\def\sumg{\sum_{K\in\mathcal{G}}}

\def\opp{\oplus^\perp}

\def\od{\mathbf{d}}

\def\odelta{\boldsymbol{\delta}}
\def\okappa{\boldsymbol{\kappa}}

\usepackage[mathscr]{eucal}

\def\fomega{\boldsymbol{\omega}}
\def\fmu{\boldsymbol{\mu}}

\def\feta{\boldsymbol{\eta}}

\def\ftau{\boldsymbol{\tau}}

\def\fpsi{\boldsymbol{\psi}}

\def\ff{\boldsymbol{f}}
\def\fW{\boldsymbol{W}}
\def\fV{\boldsymbol{V}}

\def\ixalpha{\boldsymbol{\alpha}} 
\def\ixbeta{\boldsymbol{\beta}}
\def\ixve{\boldsymbol{\varepsilon}}

\def\pddempdL{\mathbf{P}_{\od\cap\odelta}^{\rm m+\od}\Lambda}
\def\pddempdeL{\mathbf{P}_{\od\cap\odelta}^{\rm m+\odelta}\Lambda}

\def\pddedmpdL{\mathbf{P}_{\od\cap\odelta,\od}^{\rm m+\od}\Lambda}
\def\pddedempdL{\mathbf{P}_{\od\cap\odelta,\odelta}^{\rm m+\od}\Lambda}

\begin{document}

\title{A primal finite element scheme of the $\mathbf{H}(\od)\cap \mathbf{H}(\odelta)$ elliptic problem}

\author{Shuo Zhang}
\address{LSEC, Institute of Computational Mathematics and Scientific/Engineering Computing, Academy of Mathematics and System Sciences, Chinese Academy of Sciences, Beijing 100190; University of Chinese Academy of Sciences, Beijing, 100049; People's Republic of China}
\email{szhang@lsec.cc.ac.cn}

\thanks{The research is partially supported by NSFC (11871465) and CAS (XDB 41000000).}


%
%
%
%
%


\begin{abstract}
In this paper, a unified family, for $n\geqslant2$ and $1\leqslant k \leqslant n-1$, of finite element schemes are presented for the primal weak formulation of the $n$-dimensional $H(\od^k)\cap H(\odelta_k)$ elliptic problem. 
\end{abstract}

\maketitle

\tableofcontents


%
%

%
%
%
\section{Introduction}

Let $\Omega\subset\mathbb{R}^n$ be a domain with Lipschitz boundary. In this paper, we study the $H(\od)\cap H(\odelta)$ elliptic problem: given $\ff\in L^2\Lambda^k(\Omega)$, find $\fomega\in H\Lambda^k(\Omega)\cap H^*_0\Lambda^k(\Omega)$, such that 
\begin{equation}\label{eq:primalhodge}
\langle\od^k\fomega,\od^k\fmu\rangle_{L^2\Lambda^{k+1}}+\langle\odelta_k\fomega,\odelta_k\fmu\rangle_{L^2\Lambda^{k-1}}+\langle\fomega,\fmu\rangle_{L^2\Lambda^k}=\langle \ff,\fmu\rangle_{L^2\Lambda^k},\ \ \forall\,\fmu\in  H\Lambda^k(\Omega)\cap H^*_0\Lambda^k(\Omega).
\end{equation} 
Here, following \cite{Arnold.D2018feec}, we denote by $\Lambda^k(\Xi)$ the space of differential $k$-forms on an $n$-dimensional domain $\Xi$, and $L^2\Lambda^k(\Xi)$ consists of differential $k$-forms with coefficients in $L^2(\Xi)$ component by component, and $\langle\cdot,\cdot\rangle_{L^2\Lambda^k(\Xi)}$ is the inner product of the Hilbert space $L^2\Lambda^k(\Xi)$. In this paper, we will occasionally drop $\Omega$ for differential forms on $\Omega$.  The exterior differential operator $\od^k:\Lambda^k(\Xi)\to \Lambda^{k+1}(\Xi)$ is unbounded from $L^2\Lambda^k(\Xi)$ to $\Lambda^{k+1}(\Xi)$. Denote, 
$$
H\Lambda^k(\Xi)=H(\od^k,\Xi):=\left\{\fomega\in L^2\Lambda^k(\Xi):\od^k\fomega\in L^2\Lambda^{k+1}(\Xi)\right\},
$$
then $H\Lambda^k(\Xi)$ is a Hilbert space with the norm $\|\fomega\|_{L^2\Lambda^k(\Xi)}+\|\od^k\fomega\|_{L^2\Lambda^{k+1}(\Xi)}$. Denote by $H_0\Lambda^k(\Xi)$ the closure of $\mathcal{C}_0^\infty\Lambda^k(\Xi)$ in $H\Lambda^k(\Xi)$. The Hodge star operator $\star$ maps $L^2\Lambda^k(\Xi)$ isomorphically to $L^2\Lambda^{n-k}(\Xi)$ for each $0\leqslant k\leqslant n$. The \emph{codifferential operator} $\odelta_k$ defined by $\odelta_k\fmu=(-1)^{kn}\star\od^{n-k}\star\fmu$ is unbounded from $L^2\Lambda^k(\Xi)$ to $L^2\Lambda^{k-1}(\Xi)$.  Denote 
$$
H^*\Lambda^k(\Xi)=H(\odelta_k,\Xi):=\left\{\fmu\in L^2\Lambda^k(\Xi):\odelta_k\fmu\in L^2\Lambda^{k-1}(\Xi)\right\},
$$
and $H^*_0\Lambda^k(\Xi)$ the closure of $\mathcal{C}_0^\infty\Lambda^k(\Xi)$ in $H^*\Lambda^k(\Xi)$. Then $H^*\Lambda^k(\Xi)=\star H\Lambda^{n-k}(\Xi)$, and $H^*_0\Lambda^k(\Xi)=\star H_0\Lambda^{n-k}(\Xi)$. The model problem \eqref{eq:primalhodge} corresponds to a strong form that 
$$
\od^k\fomega\in H^*_0\Lambda^{k+1}(\Omega),\ \ \ \odelta_k\fomega\in H\Lambda^{k-1}(\Omega),
$$
and
\begin{equation}
\odelta_{k+1}\od^k\fomega+\od^{k-1}\odelta_k\fomega+\fomega=\ff. 
\end{equation}

The model problem arises in many applied sciences, including electromagnetics\cite{Monk.P2003mono,Hiptmair.R2002acta}, fluid-structure interaction \cite{Bathe.K;Nitikitpaiboon.C;Wang.X1995cs,Bermudez.A;Rodriguez.R1994cmame,Hamdi.M;Ousset.Y;Verchery.G1978ijnme}, and others. Particularly, known as Hodge-Laplacian operator, the finite element methods associated with $\odelta_{k+1}\od^k+\od^{k-1}\odelta_k$ have been a central topic of the finite element exterior calculus (FEEC) discussed in many aspects, and we refer to \cite{Arnold.D;Falk.R;Winther.R2006acta,Arnold.D;Falk.R;Winther.R2010bams,Arnold.D2018feec} for a thorough introduction to FEEC. 
~\\

It is well recognized that, the conforming finite element scheme for \eqref{eq:primalhodge} may lead to a spurious solution that converges to a wrong limit when the exact solution $\fomega$ is not smooth enough. Indeed, taking the $H\Lambda^1\cap H^*\Lambda^1$ in two dimension for example, a piecewise polynomial subspace of $H\Lambda^1\cap H^*_0\Lambda^1$ is contained in $H^1\Lambda^1\cap H^*_0\Lambda^1$, while $H^1\Lambda^1\cap H^*_0\Lambda^1$ is a closed subspace of $H\Lambda^1\cap H^*_0\Lambda^1$, and thus the singular part of $\fomega$ can not be captured by the conforming finite element space. To cope with this situation, a well-developed approach is to use mixed finite element method. Again, the main approach can be found in detail in \cite{Arnold.D;Falk.R;Winther.R2006acta,Arnold.D;Falk.R;Winther.R2010bams,Arnold.D2018feec}; some recent progress can be found in \cite{Li.Y2019sinum,Demlow.A;Hirani.A2014} for {\it a posteriori} error estimation and adaptive methods, and in \cite{Hong.Q;Li.Y;Xu.J2022mc} for a detailed analysis of Discontinuous Galerkin (DG) methods in FEEC in the newly-presented eXtended Galerkin (XG) framework.

On the other hand, to discretize directly the primal formulation \eqref{eq:primalhodge} has been still attracting research interests. Virtual element methods are designed for the three dimensional vector potential formulation of magnetostatic problems \cite{Beiroa.L;Brezzi.F;Marini.D;Alessandro.R2018}, with the major interests restricted to cases in which the solution is reasonably smooth and hence where higher order methods could be more profitable. Nonconforming element methods and discontinuous Galerkin methods are also designed which can lead to a correct approximation of the nonsmooth solution for the $H(\curl)\cap H(\dv)$ problem in two dimension; readers are referred to \cite{Brenner.S;Sung.L;Cui.J2008} for an interior penalty method, to \cite{brenner2009quadratic} for a nonconforming finite element method, and  to \cite{Brenner.S;Cui.J;Li.F;Sung.L2008nm} for a nonconforming finite element used with inter-element penalties. Recent works also include \cite{Barker.M2022thesis,Barker.M;Cao.S;Stern.A2022arxiv,Mirebeau.J2012aml}.
~\\

In this paper, we present a unified family of nonconforming finite element schemes for the primal formulation \eqref{eq:primalhodge} for any $n\geqslant 2$ and $1\leqslant k\leqslant n-1$, and on the subdivision of the domain by simplexes. The main feature of the finite element schemes is, all the finite element functions are defined by local shape function spaces and the continuity conditions, and no penalty term or stabilization is used in the schemes. The local shape function spaces are a slightly enrichment by $\mathcal{H}^2_{\od}(T)$ (see \eqref{eq:defhd}) based on \ $\mathcal{P}_0\Lambda^k(T)+\okappa(\mathcal{P}_0\Lambda^{k+1}(T))+\star\okappa\star(\mathcal{P}_0\Lambda^{k-1}(T))$, namely the minimal local space for $\odelta_{k+1}\od^k+\od^{k-1}\odelta_k$, but they do not contain the complete linear polynomial spaces which are used in \cite{Brenner.S;Sung.L;Cui.J2008,Brenner.S;Cui.J;Li.F;Sung.L2008nm,brenner2009quadratic,Barker.M2022thesis,Barker.M;Cao.S;Stern.A2022arxiv,Mirebeau.J2012aml}. Another difference from them, particularly \cite{brenner2009quadratic,Mirebeau.J2012aml}, is that the finite element functions in this present paper possess a different kind of inter-element continuity. As precisely described in \eqref{eq:fems}, the continuity is imposed in a dual way, and the consistency error can this way be controlled. This dual way makes the finite element functions not correspond to a ``finite element" in the sense of Ciarlet's triple \cite{Ciarlet.P1978book}, and the analysis and implementation would rely on non-standard techniques. For one thing, a stable interpolator is given, which works for functions in $H\Lambda^k(\Omega)\cap H^*_0\Lambda^k(\Omega)$ with minimal regularity, and an optimal approximation error estimation can be proved. Different from classical interpolators discussed in \cite{Arnold.D;Falk.R;Winther.R2006acta,Licht.M2019mc,Christiansen.S;Winther.R2008smoothed,Falk.R;Winther.R2014mc}, the interpolator is cell-wise defined, namely for $\fomega,\fmu\in H\Lambda^k\cap H^*_0\Lambda^k$, if $\fomega$ and $\fmu$ are equal on a cell, then their interpolations are equal on this cell. Besides, a precise set of basis functions can be  presented, the supports of which are each contained in a vertex patch, and the programming of the scheme can be done in a standard routine as for the standard ``finite element" method. Indeed, locally supported basis functions have been also found and implemented for many specific non-Ciarlet type finite element spaces \cite{Fortin.M;Soulie.M1983,Park.C;Sheen.D2003,Zhang.S2020IMA,Zhang.S2021SCM,Liu.W;Zhang.S2022arxiv,Xi.Y;Ji.X;Zhang.S2020jsc,Xi.Y;Ji.X;Zhang.S2021cicp}. 

Finally we remark that, the construction of the finite element scheme is fit for quite general situations. For example, a quite the same scheme can be constructed where the local shape function space is a slight enrichment by $\mathcal{H}^2_{\odelta}(T)$ (see \eqref{eq:defhdelta}) based on \ $\mathcal{P}_0\Lambda^k(T)+\okappa(\mathcal{P}_0\Lambda^{k+1}(T))+\star\okappa\star(\mathcal{P}_0\Lambda^{k-1}(T))$ $\pddempdeL^k(T)$. The construction is also fit for other boundary conditions, such as the elliptic problems on $H_0\Lambda^k(\Omega)\cap H^*\Lambda^k(\Omega)$.
~\\

The remaining of the paper is organized as follows. In Section \ref{sec:fes}, we define the finite element spaces and prove its optimal approximation to functions in $H\Lambda^k\cap H^*_0\Lambda^k$ by constructing a cell-wise defined interpolator. In Section \ref{sec:fescheme}, we define the finite element schemes, present the error estimation and illustrate that the scheme can practically implemented by presenting a set of locally supported basis functions.

\section{Finite element space for $H\Lambda^k(\Omega)\cap H^*_0\Lambda^k(\Omega)$}

\label{sec:fes}

%

%
%
\subsection{Polynomial spaces on a simplex}

Denote the set of $k$-indices as
$$
\mathbb{IX}_{k,n}:=\{\ixalpha=(\ixalpha_1,\dots,\ixalpha_k)\in\mathbb{N}^k:1\leqslant \ixalpha_1<\ixalpha_2<\dots<\ixalpha_k\leqslant n\},\ \ \mathbb{N}\ \mbox{the\ set\ of\ integers}.
$$

Then 
$$
\od^k(\okappa(\dx^{\ixalpha_1}\wedge\dots\wedge\dx^{\ixalpha_{k+1}}))=(k+1)\dx^{\ixalpha_1}\wedge\dots\wedge\dx^{\ixalpha_{k+1}},\ \ \mbox{for}\ \ \ixalpha\in\mathbb{IX}_{k+1,n},
$$
and
$$
\odelta_k(\star\okappa\star(\dx^{\ixalpha_1}\wedge\dots\wedge \dx^{\ixalpha_{k-1}}))=(-1)^{kn-n-1} (n-k+1)(\dx^{\ixalpha_1}\wedge\dots\wedge\dx^{\ixalpha_{k-1}}),\ \ \mbox{for}\ \ \ixalpha\in\mathbb{IX}_{k-1,n},
$$
where $\okappa$ is the Koszul operator 
$$
\okappa(\dx^{\ixalpha_1}\wedge\dots\wedge\dx^{\ixalpha_k})=\sum_{j=1}^k(-1)^{j+1}x^{\ixalpha_j}\dx^{\ixalpha_1}\wedge\dots\wedge\dx^{\ixalpha_{j-1}}\wedge\dx^{\ixalpha_{j+1}}\wedge\dots\wedge\dx^{\ixalpha_k},\ \ \mbox{for}\ \ \ixalpha\in\mathbb{IX}_{k,n}.
$$

\subsubsection{Structures of polynomials on a simplex}

Given $T$ a simplex, denote on $T$ $\tilde{x}^j=x^j-c_j$ where $c_j$ is a constant such that $\int_T\tilde{x}^j=0$. Denote a simplex dependent Koszul operator 
$$
\okappa_T(\dx^{\ixalpha_1}\wedge\dots\wedge\dx^{\ixalpha_k}):=\sum_{j=1}^k(-1)^{(j+1)}\tilde{x}^{\ixalpha_j}\dx^{\ixalpha_1}\wedge\dots\dx^{\ixalpha_{j-1}}\wedge\dx^{\ixalpha_{j+1}}\wedge\dots\wedge\dx^{\ixalpha_k},\ \ \mbox{for}\ \ \ixalpha\in\mathbb{IX}_{k,n}.
$$
Then
$$
\od^{k-1}\okappa_T(\dx^{\ixalpha_1}\wedge\dots\dx^{\ixalpha_k})=k\dx^{\ixalpha_1}\wedge\dots\dx^{\ixalpha_k}. 
$$

\begin{lemma}\label{lem:pikappa}
There exists a constant $C_{k,n}$, depending on the regularity of $T$, such that 
\begin{equation}\label{eq:pimud}
\|\fmu\|_{L^2\Lambda^k(T)}\leqslant C_{k,n}h_T\|\od^k\fmu\|_{L^2\Lambda^{k+1}(T)},\ \ \mbox{for}\ \ \fmu\in\okappa_T(\mathcal{P}_0\Lambda^{k+1}(T)),
\end{equation}
and
\begin{equation}\label{eq:pimude}
\|\fmu\|_{L^2\Lambda^k(T)}\leqslant C_{k,n}h_T\|\odelta_k\fmu\|_{L^2\Lambda^{k+1}(T)}\ \ \mbox{for}\ \fmu\in\star\okappa_T\star(\mathcal{P}_0\Lambda^{k-1}(T)).  
\end{equation}
\end{lemma}
\begin{proof}
Given $
\displaystyle \fmu=\sum_{\ixalpha\in\mathbb{IX}_{k+1,n}}C_{\ixalpha} \left(\sum_{j=1}^{k+1}(-1)^{j+1}\tilde x^{\ixalpha_j}\dx^{\ixalpha_1}\wedge\dx^{\ixalpha_2}\wedge\dots\wedge\dx^{\ixalpha_{j-1}}\wedge\dx^{\ixalpha_{j+1}}\wedge\dots\wedge \dx^{\ixalpha_{k+1}}\right)$, 
\begin{multline*}
|\fmu|^2_{H^1\Lambda^k(T)}=\left\|\sum_{\ixalpha\in\mathbb{IX}_{k,n}}C_{\ixalpha}\sum_{j=1}^{k+1}(-1)^{j+1}\nabla \tilde{x}^{\ixalpha_j}\dx^{\ixalpha_1}\wedge\dots\wedge\dx^{\ixalpha_{j-1}}\wedge\dx^{\ixalpha_{j+1}}\wedge\dots\wedge \dx^{\ixalpha_{k+1}}\right\|_{L^2\Lambda^k(T)}^2
\\
=\left\langle \sum_{\ixalpha\in\mathbb{IX}_{k,n}}C_{\ixalpha}\sum_{j=1}^{k+1}(-1)^{j+1}\nabla \tilde{x}^{\ixalpha_j}\dx^{\ixalpha_1}\wedge\dots\wedge\dx^{\ixalpha_{j-1}}\wedge\dx^{\ixalpha_{j+1}}\wedge\dots\wedge \dx^{\ixalpha_{k+1}}, \right.\qquad
\\
\qquad\qquad\left.\sum_{\ixalpha'\in\mathbb{IX}_{k,n}}C_{\ixalpha'}\sum_{i=1}^{k+1}(-1)^{i+1}\nabla \tilde{x}^{\ixalpha'_i}\dx^{\ixalpha'_1}\wedge\dots\wedge\dx^{\ixalpha'_{i-1}}\wedge\dx^{\ixalpha'_{i+1}}\wedge\dots\wedge \dx^{\ixalpha'_{k+1}}\right\rangle_{L^2\Lambda^k(T)}
\\
=\sum_{\ixalpha\in\mathbb{IX}_{k,n}}\sum_{\ixalpha'\in\mathbb{IX}_{k,n}} C_{\ixalpha}C_{\ixalpha'}\sum_{j=1}^{k+1}\sum_{i=1}^{k+1} (-1)^{j+i}e^{\ixalpha_j}\cdot e^{\ixalpha_i} \Bigg\langle \dx^{\ixalpha_1}\wedge\dots\wedge\dx^{\ixalpha_{j-1}}\wedge\dx^{\ixalpha_{j+1}}\wedge\dots\wedge \dx^{\ixalpha_{k+1}}, 
\\
\dx^{\ixalpha'_1}\wedge\dots\wedge\dx^{\ixalpha'_{i-1}}\wedge\dx^{\ixalpha'_{i+1}}\wedge\dots\wedge \dx^{\ixalpha'_{k+1}}\Bigg\rangle_{L^2\Lambda^k(T)}=(k+1)|T|\sum_\alpha C_\alpha^2,
\end{multline*}
and
$$
\left\|\od^k\mu\right\|_{L^2\Lambda^{k+1}(T)}^2=(k+1)^2\left\|\sum_\alpha C_\alpha \dx^{\ixalpha_1}\wedge\dx^{\ixalpha_2}\wedge\dots\wedge\dx^{\ixalpha_{k+1}}\right\|_{L^2\Lambda^{k+1}(T)}^2=(k+1)^2|T|\sum_{\ixalpha}C_{\ixalpha}^2.
$$
Namely
$$
\|\od^k\mu\|_{L^2\Lambda^{k+1}(T)}=\sqrt{k+1}|\fmu|_{H^1\Lambda^k(T)}.
$$
Therefore, by noting that $\int_T\tilde{x}^j=0$, with a constant $C_n$ depending on the regularity of $T$, we obtain
$$
\|\fmu\|_{L^2\Lambda^k(T)}\leqslant C_nh_T|\fmu|_{H^1\Lambda^k(T)}=C_n(k+1)^{-1/2}h_T\|\od^k\mu\|_{L^2\Lambda^{k+1}(T)}.
$$
This proves \eqref{eq:pimud}. Similarly can \eqref{eq:pimude} be proved. 
\end{proof}

In this part and in the sequel, we make the convention that, for $\ixalpha\in\mathbb{IX}_{k,n}$, we use $\ixbeta$ for one in $\mathbb{IX}_{n-k,n}$, such that $\ixalpha$ and $\ixbeta$ partition $\{1,2,\dots,n\}$. For $\ixalpha\in\mathbb{IX}_{k,n}$, denote 
$$
\tilde\fmu_{\odelta,T}^{\ixalpha}=\sum_{j=1}^k[(\tilde{x}^{\ixalpha_j})^2-c^{\ixalpha_j}]\dx^{\ixalpha_1}\wedge\dots\wedge\dx^{\ixalpha_k}, 
$$
and
$$
\tilde\fmu_{\od,T}^{\ixalpha}=\sum_{j=1}^{n-k}[(\tilde{x}^{\ixbeta_j})^2-c^{\ixbeta_j}]\dx^{\ixalpha_1}\wedge\dx^{\ixalpha_2}\wedge\dots\wedge\dx^{\ixalpha_k}, 
$$
where $c^{\ixalpha_j}$ and $c^{\ixbeta_j}$ are constants such that $\int_T [(\tilde x^{\ixalpha_j})^2-c^{\ixalpha_j}]=0$ and $\int_T[(\tilde x^{\ixbeta_j})^2-c^{\ixbeta_j}]=0$. Then
\begin{equation}
\od^k\tilde\fmu_{\odelta,T}^{\ixalpha}=0,\quad \odelta_k\tilde\fmu_{\odelta,T}^{\ixalpha}=(-1)^n\cdot 2\okappa_T(\dx^{\ixalpha_1}\wedge\dx^{\ixalpha_2}\wedge\dots\wedge \dx^{\ixalpha_k}),
\end{equation}

\begin{equation}
\odelta_k\tilde\fmu_{\od,T}^{\ixalpha}=0,\quad \mbox{and},\quad \od^k\tilde\fmu_{\od,T}^{\ixalpha}=2(-1)^{n(1+k)+1}\star(\okappa_T(\star (\dx^{\ixalpha_1}\wedge\dots\dx^{\ixalpha_k}))).
\end{equation}
We refer to the appendix, particularly \eqref{eq:Aodelta} and \eqref{eq:Aod}, for some detailed calculations. 

Denote 
\begin{equation}\label{eq:defhd}
\mathcal{H}^2_{\od}\Lambda^k(T):={\rm span}\{\tilde\fmu_{\od,T}^{\ixalpha}:\ixalpha\in\mathbb{IX}_{k,n}\},
\end{equation}
and
\begin{equation}\label{eq:defhdelta}
\mathcal{H}^2_{\odelta}\Lambda^k(T):={\rm span}\{\tilde\fmu_{\odelta,T}^{\ixalpha}:\ixalpha\in\mathbb{IX}_{k,n}\}.
\end{equation}

\begin{lemma}\label{lem:surjecdde}
\begin{enumerate}
\item $\od^k$ is bijective from $\okappa_T(\mathcal{P}_0\Lambda^{k+1})$ onto $\mathcal{P}_0\Lambda^{k+1}$, and bijective from $\mathcal{H}^2_{\od}\Lambda^k(T)$ onto $\star\okappa_T\star(\mathcal{P}_0\Lambda^k(T))$.
\item $\odelta_k$ is bijective from $\star\okappa_T\star(\mathcal{P}_0\Lambda^{k-1})$ onto $\mathcal{P}_0\Lambda^{k-1}$, and bijective from $\mathcal{H}^2_{\odelta}\Lambda^k(T)$ onto $\okappa_T(\mathcal{P}_0\Lambda^k(T))$.
\end{enumerate}
\end{lemma}

\begin{lemma}\label{lem:pih}
There exists a constant $C_{k,n}$, depending on the regularity of $T$, such that
\begin{equation}\label{eq:hfhde}
\|\fmu\|_{L^2\Lambda^k(T)}\leqslant C_{k,n}h_T\|\odelta_k\fmu\|_{L^2\Lambda^{k-1}(T)},\ \ \mbox{for}\ \fmu\in \mathcal{H}^2_{\odelta}\Lambda^k(T),
\end{equation}
and
\begin{equation}\label{eq:hfhd}
\|\fmu\|_{L^2\Lambda^k(T)}\leqslant C_{k,n}h_T\|\od^k\fmu\|_{L^2\Lambda^{k+1}(T)},\ \ \mbox{for}\ \fmu\in \mathcal{H}^2_{\od}\Lambda^k(T).
\end{equation}
\end{lemma}
\begin{proof}
For $\displaystyle\fmu=\sum_{\ixve\in\mathbb{IX}_{k,n}}C_{\ixve}\tilde\fmu_{\odelta,T}^{\ixve}$,
$$
\|\fmu\|_{L^2\Lambda^k(T)}^2\leqslant C_{k,n}h_T^2\|\sum_{\ixve}C_{\ixve}\sum_{1\leqslant j\leqslant k}\nabla(\tilde{x}^{\ixve_j})^2\dx^{\ixve_1}\wedge\dx^{\ixve_2}\wedge\dots\wedge\dx^{\ixve_k}\|_{L^2\Lambda^k(T)}^2\leqslant C_{k,n}h_T^4|T|\sum_{\ixve}C_{\ixve}^2.
$$
Note that
$$
\odelta_k\fmu=2(-1)^n\sum_{\ixve}C_{\ixve}\okappa_T(\dx^{\ixve_1}\wedge\dots\wedge\dx^{\ixve_k})
$$
and 
$$
\od^{k-1}\odelta_k\fmu=2k(-1)^n\sum_{\ixve\in\mathbb{IX}_{k,n}}C_{\ixve}\dx^{\ixve_1}\wedge\dots\wedge\dx^{\ixve_k}.
$$
Therefore, by the inverse inequality,
$$
4k^2h_T^2|T|\sum_{\ixve\in\mathbb{IX}_{k,n}}C_{\ixve}^2=h_T^2\|\od^{k-1}\odelta_k\fmu\|_{L^2\Lambda^k(T)}^2\leqslant C_{k,n}\|\odelta_k\fmu\|_{L^2\Lambda^{k-1}(T)}^2.
$$
The proof of \eqref{eq:hfhde} is thus completed. Similarly can \eqref{eq:hfhd} be proved. 
\end{proof}

\begin{lemma}
There exists a constant $C_{k,n}$, depending on the regularity of $T$, such that
\begin{equation}
\|\fmu\|_{L^2\Lambda^k(T)}\leqslant C_{k,n}h_T\|\odelta_k\fmu\|_{L^2\Lambda^{k-1}(T)},\ \ \mbox{for}\ \fmu\in \star\okappa_T\star(\mathcal{P}_0\Lambda^{k-1}(T))+\mathcal{H}^2_{\odelta}\Lambda^k(T),
\end{equation}
and
\begin{equation}
\|\fmu\|_{L^2\Lambda^k(T)}\leqslant C_{k,n}h_T\|\od^k\fmu\|_{L^2\Lambda^{k+1}(T)},\ \ \mbox{for}\ \fmu\in \okappa_T(\mathcal{P}_0\Lambda^{k+1}(T))+\mathcal{H}^2_{\od}\Lambda^k(T).
\end{equation}
\end{lemma}
\begin{proof}
Note that $\odelta_k(\star\okappa_T\star(\mathcal{P}_0\Lambda^{k-1}(T)))$ and $\odelta_k(\mathcal{H}^2_{\odelta}\Lambda^k(T))$ are orthogonal, and $\od^k(\okappa_T(\mathcal{P}_0\Lambda^{k+1}(T)))$ and  $\od^k(\mathcal{H}^2_{\od}\Lambda^k(T))$ are orthogonal. The lemma follows by Lemmas \ref{lem:pikappa} and \ref{lem:pih}. 
\end{proof}

It is well known the lowest-degree trimmed space is
$$
\mathcal{P}^-_1\Lambda^k(T)=\mathcal{P}_0\Lambda^k(T)+\okappa(\mathcal{P}_0\Lambda^{k+1}(T)),
$$
and denote 
$$
\mathcal{P}^{*,-}_1\Lambda^k(T)=\star(\mathcal{P}^-_1\Lambda^k(T))=\mathcal{P}_0\Lambda^k(T)+\star\okappa\star(\mathcal{P}_0\Lambda^{k-1}(T)).
$$
Here we introduce, for the $H\Lambda^k\cap H^*_0\Lambda^k$ problem, an enriched trimmed space, defined by
\begin{equation}
\pddempdL^k(T):=\mathcal{P}_0\Lambda^k(T)\oplus\okappa_T(\mathcal{P}_0\Lambda^{k+1}(T))\oplus\star\okappa_T\star(\mathcal{P}_0\Lambda^{k-1}(T))\oplus\mathcal{H}^2_{\od}\Lambda^k(T).
\end{equation}

Then 
$$
\od^k(\pddempdL^k(T))=\mathcal{P}^{*,-}_1\Lambda^{k+1}(T),\ \ \mbox{and}\ \ \odelta_k(\pddempdL^k(T))=\mathcal{P}_0\Lambda^{k-1}(T).
$$

\subsubsection{A projective interpolator}

Define the interpolator
$$
\mathbb{I}^{\rm m+\od,k}_{\od\cap\odelta,T}:H\Lambda^k(T)\cap H^*\Lambda^k(T)\to \pddempdL^k(T),
$$
such that, for $\fmu\in H\Lambda^k(T)\cap H^*\Lambda^k(T)$,
\begin{multline}\label{eq:int-1}
\langle\od^k \mathbb{I}^{\rm m+\od,k}_{\od\cap\odelta,T}\fmu,\feta\rangle_{L^2\Lambda^{k+1}(T)}-\langle \mathbb{I}^{\rm m+\od,k}_{\od\cap\odelta,T}\fmu, \odelta_{k+1} \feta\rangle_{L^2\Lambda^k(T)}
\\
=
\langle\od^k \fmu,\feta\rangle_{L^2\Lambda^{k+1}(T)}-\langle \fmu, \odelta_{k+1} \feta\rangle_{L^2\Lambda^k(T)},\ \forall\,\feta\in\mathcal{P}^{*,-}_1\Lambda^{k+1}(T),
\end{multline}
and
\begin{multline}\label{eq:int-2}
\langle\odelta_k \mathbb{I}^{\rm m+\od,k}_{\od\cap\odelta,T}\fmu,\ftau\rangle_{L^2\Lambda^{k-1}(T)}-\langle \mathbb{I}^{\rm m+\od,k}_{\od\cap\odelta,T}\fmu, \od^{k-1} \ftau\rangle_{L^2\Lambda^k(T)}
\\
=
\langle\odelta_k \fmu,\ftau\rangle_{L^2\Lambda^{k-1}(T)}-\langle \fmu, \od^{k-1} \ftau\rangle_{L^2\Lambda^k(T)},\ \forall\,\ftau\in\mathcal{P}^-_1\Lambda^{k-1}(T).
\end{multline}

\begin{lemma}\label{lem:welldintpprojec}
The interpolator $\mathbb{I}^{\rm m+\od,k}_{\od\cap\odelta,T}$ is well defined, and $\mathbb{I}^{\rm m+\od,k}_{\od\cap\odelta,T}\fmu=\fmu$ for $\fmu\in \pddempdL^k(T)$. 
\end{lemma}
\begin{proof}
Elementarily, 
$$
\mathcal{P}^{*,-}_1\Lambda^{k+1}(T)=\mathcal{P}_0\Lambda^{k+1}(T)\opp\star\okappa_T\star(\mathcal{P}_0\Lambda^{k}(T)),
$$
and 
$$
\mathcal{P}^-_1\Lambda^{k-1}(T)=\mathcal{P}_0\Lambda^{k-1}(T)\opp\okappa_T(\mathcal{P}_0\Lambda^k(T)). 
$$
Given $\feta\in \mathcal{P}^{*,-}_1\Lambda^{k+1}(T)$, decompose it to $\feta=\feta_0+\feta_{\odelta}$, with $\feta_0\in \mathcal{P}_0\Lambda^{k+1}(T)$ and $\feta_{\odelta}\in \star\okappa_T\star(\mathcal{P}_0\Lambda^{k}(T))$, and decompose $\ftau\in \mathcal{P}^-_1\Lambda^{k-1}(T)$ to $\ftau=\ftau_0+\ftau_{\od}$ with $\ftau_0\in\mathcal{P}_0\Lambda^{k-1}(T)$ and $\ftau_{\od}\in \okappa_T(\mathcal{P}_0\Lambda^k(T))$. 

Now, given $\fmu\in H\Lambda^k(T)\cap H^*\Lambda^k(T)$, there exists a unique $\fmu_{\odelta}\in \star\okappa_T\star(\mathcal{P}_0\Lambda^{k-1}(T))$, such that 
$$
\langle\odelta_k\fmu_{\odelta},\ftau_0\rangle_{L^2\Lambda^{k-1}(T)}=\langle\odelta_k\fmu,\ftau_0\rangle_{L^2\Lambda^{k-1}(T)},\ \forall\,\ftau_0\in\mathcal{P}_0\Lambda^{k-1}(T),
$$
and there exists a unique $\fmu_0\in\mathcal{P}_0\Lambda^k(T)$, such that 
$$
-\langle\fmu_0,\od^{k-1}\ftau_{\od}\rangle_{L^2\Lambda^k(T)}=\langle\odelta_k\fmu,\ftau_{\od}\rangle_{L^2\Lambda^{k-1}(T)}-\langle\fmu,\od^{k-1}\ftau_{\od}\rangle_{L^2\Lambda^k(T)},\ \forall\,\ftau_{\od}\in \okappa_T(\mathcal{P}_0\Lambda^k(T)).
$$
Then, there exists a unique $\fmu_{\od}\in \okappa_T(\mathcal{P}_0\Lambda^{k+1}(T))$, such that 
\begin{equation}\label{eq:dkmueta0}
\langle \od^k\fmu_{\od},\feta_0\rangle_{L^2\Lambda^{k+1}(T)}=\langle\od^k\fmu,\feta_0\rangle_{L^2\Lambda^{k+1}(T)},\ \forall\,\feta_0\in\mathcal{P}_0\Lambda^{k+1}(T),
\end{equation}
and there exists a unique $\fmu_{\od+}\in\mathcal{H}^2_{\od}\Lambda^k(T)$, such that 
\begin{multline}\label{eq:dkmuetaod}
\langle\od^k\fmu_{\od+},\feta_{\odelta}\rangle_{L^2\Lambda^{k+1}(T)}-\langle\fmu_0,\odelta_{k+1}\feta_{\odelta}\rangle_{L^2\Lambda^k(T)}
\\
=\langle\od^k\fmu,\feta_{\odelta}\rangle_{L^2\Lambda^{k-1}(T)}-\langle\fmu,\odelta_{k+1}\feta_{\odelta}\rangle_{L^2\Lambda^k(T)},\ \forall\,\feta_{\odelta}\in\star\okappa_T\star(\mathcal{P}_0\Lambda^k(T)). 
\end{multline}
Now set $\fmu_K:=\fmu_0+\fmu_{\od}+\fmu_{\odelta}+\fmu_{\od+}$, and $\fmu_K$ satisfies all requirements \eqref{eq:int-1} and \eqref{eq:int-2} of $\mathbb{I}^{\rm m+\od,k}_{\od\cap\odelta,T}\fmu$; namely, $\mathbb{I}^{\rm m+\od,k}_{\od\cap\odelta,T}\fmu=\fmu_0+\fmu_{\od}+\fmu_{\odelta}+\fmu_{\od+}$, uniquely determined. Evidently, if $\fmu\in \pddempdL^k(T)$, $\fmu=\mathbb{I}^{\rm m+\od,k}_{\od\cap\odelta,T}\fmu$. The proof is completed. 
\end{proof}

\begin{remark}\label{rem:nodalf}
Denote a set of quantities 
\begin{equation}\label{eq:detcond}
\left\{\langle\od^k \fmu,\feta\rangle_{L^2\Lambda^{k+1}(T)}-\langle \fmu, \odelta_{k+1} \feta\rangle_{L^2\Lambda^k(T)},\ \ \langle\odelta_k \fmu,\ftau\rangle_{L^2\Lambda^{k-1}(T)}-\langle \fmu, \od^{k-1} \ftau\rangle_{L^2\Lambda^k(T)}\right\}.
\end{equation}
Then, according to the proof of Lemma \ref{lem:welldintpprojec}, 
\begin{itemize}
\item given $\fmu\in \pddempdL^k(T)$, $\fmu=0$ if and only if all quantities in \eqref{eq:detcond} vanish for any $(\feta,\ftau)\in \mathcal{P}^{*,-}_1\Lambda^{k+1}(T)\times\mathcal{P}^-_1\Lambda^{k-1}(T)$; 
\item given $(\feta,\ftau)\in \mathcal{P}^{*,-}_1\Lambda^{k+1}(T)\times\mathcal{P}^-_1\Lambda^{k-1}(T)$, $(\feta,\ftau)=(0,0)$, if and only if all quantities in \eqref{eq:detcond} vanish for any $\fmu\in \pddempdL^k(T)$.
\end{itemize}
This way, $\fmu\in\pddempdL^k(T)$ is unisolvent by \eqref{eq:detcond} with $(\feta,\ftau)\in \mathcal{P}^{*,-}_1\Lambda^{k+1}(T)\times\mathcal{P}^-_1\Lambda^{k-1}(T)$, and  $(\feta,\ftau)\in \mathcal{P}^{*,-}_1\Lambda^{k+1}(T)\times\mathcal{P}^-_1\Lambda^{k-1}(T)$ is unisolvent by \eqref{eq:detcond} with $\fmu\in \pddempdL^k(T)$.
\end{remark}

\begin{lemma}
There exists a constant $C_{k,n}$, depending on the regularity of the simplex $T$, such that,
\begin{multline*}
\|\mathbb{I}^{\rm m+\od,k}_{\od\cap\odelta,T}\fmu\|_{L^2\Lambda^k(T)} + \|\od^k\mathbb{I}^{\rm m+\od,k}_{\od\cap\odelta,T}\fmu\|_{L^2\Lambda^{k+1}(T)} + 
\|\odelta_k\mathbb{I}^{\rm m+\od,k}_{\od\cap\odelta,T}\fmu\|_{L^2\Lambda^{k-1}(T)}
\\
\leqslant C_{k,n}  (\|\fmu\|_{L^2\Lambda^k(T)} + \|\od^k\fmu\|_{L^2\Lambda^{k+1}(T)} + \|\odelta_k\fmu\|_{L^2\Lambda^{k-1}(T)}).
\end{multline*}
\end{lemma}

\begin{proof}
Given $\fmu\in H\Lambda^k(T)\cap H^*\Lambda^k(T)$, for the interpolation, we use $\fmu_0$, $\fmu_{\od}$, $\fmu_{\odelta}$ and $\fmu_{\od+}$ as in the proof of Lemma \ref{lem:welldintpprojec}, and $\mathbb{I}^{\rm m+\od,k}_{\od\cap\odelta,T}\fmu=\fmu_0+\fmu_{\od}+\fmu_{\odelta}+\fmu_{\od+}$. Then 
$$
\odelta_k \mathbb{I}^{\rm m+\od,k}_{\od\cap\odelta,T}\fmu=\odelta_k\fmu_{\odelta}=\mathbf{P}^k_0 \odelta_k\fmu. 
$$
Further,
$$
\langle \fmu-\fmu_0,\od^{k-1}\ftau_{\od}\rangle_{L^2\Lambda^k(T)}=\langle\odelta_k\fmu,\ftau_{\od}\rangle_{L^2\Lambda^{k-1}(T)},\ \forall\,\ftau_{\od}\in\okappa_T(\mathcal{P}_0\Lambda^k).
$$
Then, by Lemmas \ref{lem:surjecdde} and \ref{lem:pikappa}, we have
$$
\|\mathbf{P}^k_0\fmu-\fmu_0\|_{L^2\Lambda^k(T)}\leqslant C_{k,n}h_T\|\odelta_k\fmu\|_{L^2\Lambda^{k-1}(T)}.
$$
By \eqref{eq:dkmueta0} and \eqref{eq:dkmuetaod}, we have 
$$
\langle \od^k\mathbb{I}^{\rm m+\od,k}_{\od\cap\odelta,T}\fmu-\od^k\fmu,\feta_0\rangle_{L^2\Lambda^{k+1}(T)}=0,\ \forall\,\feta_0\in\mathcal{P}_0\Lambda^{k+1}(T),
$$
and
$$
\langle\od^k\mathbb{I}^{\rm m+\od,k}_{\od\cap\odelta,T}\fmu-\od^k\fmu,\feta_{\odelta}\rangle_{L^2\Lambda^{k+1}(T)}=\langle\fmu_0-\fmu,\odelta_{k+1}\feta_{\odelta}\rangle_{L^2\Lambda^k(T)},\ \forall\,\feta_{\odelta}\in\star\okappa_T\star(\mathcal{P}_0\Lambda^k(T)). 
$$
Therefore, by Lemma \ref{lem:pikappa},
$$
\|\mathbf{P}^{k+1}_{\mathcal{P}^{*,-}_1\Lambda^{k+1}(T)}\od^k\fmu-\od^k\mathbb{I}^{\rm m+\od,k}_{\od\cap\odelta,T}\fmu\|_{L^2\Lambda^{k+1}(T)}\leqslant C_{k,n}h_T^{-1}\|\mathbf{P}^k_0\fmu-\fmu_0\|_{L^2\Lambda^k(T)}\leqslant C_{k,n}\|\odelta_k\fmu\|_{L^2\Lambda^{k-1}(T)}.
$$
Summing all above leads to the assertion and completed the proof. 
\end{proof}

\begin{lemma}\label{lem:approxcell}
There exists a constant $C_{k,n}$, depending on the regularity of the simplex $T$, such that,
\begin{multline}
\|\fomega-\mathbb{I}^{\rm m+\od,k}_{\od\cap\odelta,T}\fomega\|_{L^2\Lambda^k(T)} + \|\od^k\fomega-\od^k\mathbb{I}^{\rm m+\od,k}_{\od\cap\odelta,T}\fomega\|_{L^2\Lambda^{k+1}(T)} + 
\|\odelta_k\fomega-\odelta_k\mathbb{I}^{\rm m+\od,k}_{\od\cap\odelta,T}\fomega\|_{L^2\Lambda^{k-1}(T)}
\\
\leqslant C_{k,n} \inf_{\fmu\in\pddempdL^k(T)}\left[\|\fomega-\fmu\|_{L^2\Lambda^k(T)} + \|\od^k(\fomega-\fmu)\|_{L^2\Lambda^{k+1}(T)} + \|\odelta_k(\fomega-\fmu)\|_{L^2\Lambda^{k-1}(T)}\right].
\end{multline}
\end{lemma}
\begin{proof}
The proof follows immediately by the projection and stability. 
\end{proof}

\begin{remark}
It can be proved that 
\begin{multline*}
\inf_{\fmu\in\pddempdL^k(T)}\|\fomega-\fmu\|_{L^2\Lambda^k(T)} + \|\od^k(\fomega-\fmu)\|_{L^2\Lambda^{k+1}(T)} + \|\odelta_k(\fomega-\fmu)\|_{L^2\Lambda^{k-1}(T)}
\\
\leqslant \inf_{\fmu\in\mathcal{P}_0\Lambda^k(T)}\|\fomega-\fmu\|_{L^2\Lambda^k(T)} + \inf_{\feta\in\mathcal{P}^{*,-}_1\Lambda^{k+1}(T)}\|\od^k\fomega-\feta\|_{L^2\Lambda^{k+1}(T)} + \inf_{\ftau\in\mathcal{P}_0\Lambda^{k-1}(T)}\|\odelta_k\fomega-\ftau\|_{L^2\Lambda^{k-1}(T)}
\\
+Ch_T(\|\od^k\fomega\|_{L^2\Lambda^{k+1}(T)}+\|\odelta_k\fomega\|_{L^2\Lambda^{k-1}(T)}).
\end{multline*}
\end{remark}

\subsection{Finite element space and global approximation}

For $\Xi$ a subdomain of $\Omega$, we denote by $E_\Xi^\Omega$ the extension from $L^1_{\rm loc}(\Xi)$ to $L^1_{\rm loc}(\Omega)$, the spaces of locally integrable functions, respectively. Namely, 
$$
E_\Xi^\Omega: L^1_{\rm loc}(\Xi)\to L^1_{\rm loc}(\Omega),\ \ 
E_\Xi^\Omega v=\left\{\begin{array}{ll}
v,&\mbox{on}\,\Xi,
\\
0,&\mbox{else},
\end{array}\right.
\ \mbox{for}\,v\in L^1_{\rm loc}(\Xi).
$$
We use the same notation $L^1_{\rm loc}$ for both scalar and non-scalar locally integrable functions, and, here and in the sequel, use the same notation $E_\Xi^\Omega$ for both scalar and non-scalar functions. 

Let $\mathcal{G}_\Omega=\left\{\mathcal{G}_h\right\}$ be a set of shape regular subdivisions of $\Omega$ by simplexes. For a subdivision $\mathcal{G}_h$, define formally the product of a set of function spaces $\left\{\Upsilon(T)\right\}_{T\in\mathcal{G}_h}$ defined cell by cell such that $E_T^\Omega\Upsilon(T)$ for all $T\in\mathcal{G}_h$ are compatible, 
$$
\prodg\Upsilon(T):=\sumg E_T^\Omega\Upsilon(T),
$$
and the summation is direct. $\displaystyle\prodg\Upsilon(T)$ defined this way is actually the tensor product of all $\Upsilon(T)$.

Denote
\begin{itemize}
\item $\displaystyle\mathcal{P}^-_1\Lambda^k(\mathcal{G}_h):=\prod_{T\in\mathcal{G}_h}\mathcal{P}^-_1\Lambda^k(T)$;

\item $\displaystyle\mathcal{P}^{*,-}_1\Lambda^k(\mathcal{G}_h):=\prod_{T\in\mathcal{G}_h}\mathcal{P}^{*,-}_1\Lambda^k(T)$;

\item $\displaystyle\pddempdL^k(\mathcal{G}_h):=\prod_{T\in\mathcal{G}_h}\pddempdL^k(T)$.
\end{itemize}
Denote the conforming space by Whitney forms by
$$
\fW^*_{h0}\Lambda^{k+1}=\mathcal{P}^{*,-}_1\Lambda^{k+1}(\mathcal{G}_h)\cap H^*_0\Lambda^{k+1},\ \mbox{and}\ \ 
\fW_h\Lambda^{k-1}=\mathcal{P}^-_1\Lambda^{k-1}(\mathcal{G}_h)\cap H\Lambda^{k-1}.
$$
On $\mathcal{G}_h$, define the finite element space:
\begin{multline}\label{eq:fems}
\fV_{\od\cap\odelta}^{\rm m, +\od}\Lambda^k:=\Big\{\fmu_h\in\pddempdL^k(\mathcal{G}_h):
\langle\od^k_h\fmu_h,\feta_h\rangle_{L^2\Lambda^{k+1}}-\langle\fmu_h,\odelta_{k+1}\feta_h\rangle_{L^2\Lambda^k}=0,\ \forall\,\feta_h\in\fW^*_{h0}\Lambda^{k+1},
\\
\mbox{and}\ \ \langle\odelta_{k,h}\fmu_h,\ftau_h\rangle_{L^2\Lambda^{k-1}}-\langle\fmu_h,\od^{k-1}\ftau_h\rangle_{L^2\Lambda^k}=0,\ \forall\,\ftau_h\in\fW_h\Lambda^{k-1}
\Big\}.
\end{multline}
Here and in the sequel we use the subscript $\cdot_h$ to denote the piecewise operation on $\mathcal{G}_h$.

Define the interpolator 
$$
\mathbb{I}^{\rm m+\od,k}_{\od\cap\odelta,h}:\prod_{T\in\mathcal{G}_h}H\Lambda^k(T)\cap H^*\Lambda^k(T)\to \prod_{T\in\mathcal{G}_h} \pddempdL^k(T),
$$
such that 
$$
(\mathbb{I}^{\rm m+\od,k}_{\od\cap\odelta,h}\fmu)|_T=\mathbb{I}^{\rm m+\od,k}_{\od\cap\odelta,T}(\fmu|_T),\ \forall\,T\in\mathcal{G}_h. 
$$

\begin{lemma}
If $\fmu\in H\Lambda^k\cap H^*_0\Lambda^k$, then \ \ $\mathbb{I}^{\rm m+\od,k}_{\od\cap\odelta,h}\fmu\in \fV_{\od\cap\odelta}^{\rm m, +\od}\Lambda^k$.
\end{lemma}
\begin{proof}
Provided $\fmu\in H\Lambda^k\cap H^*_0\Lambda^k$, it holds that 
$$
\left\{
\begin{array}{ll}
\langle\od^k\fmu,\feta_h\rangle_{L^2\Lambda^{k+1}}-\langle\fmu,\odelta_{k+1}\feta_h\rangle_{L^2\Lambda^k}=0,&\forall\,\feta_h\in\fW^*_{h0}\Lambda^{k+1}
\\
\langle\odelta_k\fmu,\ftau_h\rangle_{L^2\Lambda^{k-1}}-\langle\fmu,\od^{k-1}\ftau_h\rangle_{L^2\Lambda^k}=0,&\forall\,\ftau_h\in\fW_h\Lambda^{k-1}
\end{array}
\right. .
$$
By the definitions of $\mathbb{I}^{\rm m+\od,k}_{\od\cap\odelta,h}$ and $\mathbb{I}^{\rm m+\od,k}_{\od\cap\odelta,T}$, it holds that
\begin{multline*}
\langle\od^k_h\mathbb{I}^{\rm m+\od,k}_{\od\cap\odelta,h}\fmu,\feta_h\rangle_{L^2\Lambda^{k+1}}-\langle \mathbb{I}^{\rm m+\od,k}_{\od\cap\odelta,h}\fmu,\odelta_{k+1}\feta_h\rangle_{L^2\Lambda^k}
=\sum_{T\in\mathcal{G}_h}\langle\od^k\fmu,\feta_h\rangle_{L^2\Lambda^{k+1}(T)}-\langle\fmu,\odelta_{k+1}\feta_h\rangle_{L^2\Lambda^k(T)}
\\
=\langle\od^k\fmu,\feta_h\rangle_{L^2\Lambda^{k+1}(\Omega)}-\langle\fmu,\odelta_{k+1}\feta_h\rangle_{L^2\Lambda^k(\Omega)}=0,\forall\,\feta_h\in\fW^*_{h0}\Lambda^{k+1},
\end{multline*}
and similarly
$$
\langle\odelta_{k,h}\mathbb{I}^{\rm m+\od,k}_{\od\cap\odelta,h}\fmu,\ftau_h\rangle_{L^2\Lambda^{k-1}}-\langle \mathbb{I}^{\rm m+\od,k}_{\od\cap\odelta,h}\fmu,\od^{k-1}\ftau_h\rangle_{L^2\Lambda^k}=0,\ \forall\,\ftau_h\in\fW_h\Lambda^{k-1}. 
$$
Namely, $\mathbb{I}^{\rm m+\od,k}_{\od\cap\odelta,h}\fmu\in \fV_{\od\cap\odelta}^{\rm m, +\od}\Lambda^k$. The proof is completed. 
\end{proof}

By Lemma \ref{lem:approxcell}, we immediately obtain the lemma below.
\begin{lemma}\label{lem:globalinterror}
For $\fomega\in H\Lambda^k\cap H^*_0\Lambda^k$, with a constant $C_{k,n}$ depending on the shape regularity of $\mathcal{G}_h$, 
\begin{multline}
\|\fomega-\mathbb{I}^{\rm m+\od,k}_{\od\cap\odelta,h}\fomega\|_{L^2\Lambda^k} + \|\od^k\fomega-\od^k_h\mathbb{I}^{\rm m+\od,k}_{\od\cap\odelta,h}\fomega\|_{L^2\Lambda^{k+1}} + 
\|\odelta_k\fomega-\odelta_{k,h}\mathbb{I}^{\rm m+\od,k}_{\od\cap\odelta,h}\fomega\|_{L^2\Lambda^{k-1}}
\\
\leqslant C_{k,n}  \inf_{\fmu_h\in \pddempdL^k(\mathcal{G}_h)}\left[\|\fomega-\fmu_h\|_{L^2\Lambda^k} + \|\od^k_h(\fomega-\fmu_h)\|_{L^2\Lambda^{k+1}} + \|\odelta_{k,h}(\fomega-\fmu_h)\|_{L^2\Lambda^{k-1}}\right].
\end{multline}
\end{lemma}

\section{A primal finite element scheme of the $\mathbf{H}(\od)\cap \mathbf{H}(\odelta)$ elliptic problem}
\label{sec:fescheme}

\subsection{Finite element scheme and error estimate}

Given a simplicial subdivision $\mathcal{G}_h$ of $\Omega$, we consider the finite element problem: find $\fomega_h\in \fV_{\od\cap\odelta}^{\rm m, +\od}\Lambda^k$, such that 
\begin{equation}\label{eq:primalhodgedis}
\langle\od^k_h\fomega_h,\od^k_h\fmu_h\rangle_{L^2\Lambda^{k+1}}+\langle\odelta_{k,}\fomega_h,\odelta_{k,h}\fmu_h\rangle_{L^2\Lambda^{k-1}}+\langle\fomega_h,\fmu_h\rangle_{L^2\Lambda^k}=\langle \ff,\fmu_h\rangle_{L^2\Lambda^k},\  \forall\,\fmu_h\in \fV_{\od\cap\odelta}^{\rm m, +\od}\Lambda^k.
\end{equation}
The well-posed-ness of \eqref{eq:primalhodgedis} is immediate. The main theoretical result is the theorem below. 

\begin{theorem}
Let $\fomega$ and $\fomega_h$ be the solutions of \eqref{eq:primalhodge} and \eqref{eq:primalhodgedis}, respectively. Then
\begin{multline}
\|\fomega-\fomega_h\|_{L^2\Lambda^k}+\|\od^k_h(\fomega-\fomega_h)\|_{L^2\Lambda^{k+1}}+\|\odelta_{k,h}(\fomega-\fomega_h)\|_{L^2\Lambda^{k-1}}
\\
\leqslant C\inf_{\fmu_h\in \pddempdL^k(\mathcal{G}_h)}\left[\|\fomega-\fmu_h\|_{L^2\Lambda^k}+\|\od^k_h(\fomega-\fmu_h)\|_{L^2\Lambda^{k+1}}+\|\odelta_{k,h}(\fomega-\fmu_h)\|_{L^2\Lambda^{k-1}}\right]
\\
+\inf_{\feta_h\in \fW^*_{h0}\Lambda^{k+1}}\left[\|\od^k\fomega-\feta_h\|_{L^2\Lambda^{k+1}}+\|\odelta_{k+1}(\od^k\fomega-\feta_h)\|_{L^2\Lambda^k}\right]
\\
+\inf_{\ftau_h\in \fW_h\Lambda^{k-1}}\left[\|\odelta_k\fomega-\ftau_h\|_{L^2\Lambda^{k-1}}+\|\od^{k-1}(\odelta_k\fomega-\ftau_h)\|_{L^2\Lambda^k}\right].
\end{multline}
\end{theorem}
\begin{proof}
According to Strang's lemma, the consistency error part is
\begin{multline*}
CE(\fmu_h)=\langle\od^k\fomega,\od^k_h\fmu_h\rangle_{L^2\Lambda^{k+1}}+\langle\odelta_k\fomega,\odelta_{k,h}\fmu_h\rangle_{L^2\Lambda^{k-1}}+\langle\fomega,\fmu_h\rangle_{L^2\Lambda^k}-\langle \ff,\fmu_h\rangle_{L^2\Lambda^k}
\\
=\langle\od^k\fomega,\od^k_h\fmu_h\rangle_{L^2\Lambda^{k+1}}-\langle\odelta_{k+1}\od^k\fomega,\fmu_h\rangle_{L^2\Lambda^k}+\langle\odelta_k\fomega,\odelta_{k,h}\fmu_h\rangle_{L^2\Lambda^{k-1}}-\langle\od^{k-1}\odelta_k\fomega,\fmu_h\rangle_{L^2\Lambda^k}.
\end{multline*}
By the definition of $\fV_{\od\cap\odelta}^{\rm m, +\od}\Lambda^k$, we have for any $\feta_h\in \fW^*_{h0}\Lambda^{k+1}$ and $\ftau_h\in \fW_h\Lambda^{k-1}$,
\begin{multline}
CE(\fmu_h)=\langle\od^k\fomega-\feta_h,\od^k_h\fmu_h\rangle_{L^2\Lambda^{k+1}}-\langle\odelta_{k+1}(\od^k\fomega-\feta_h),\fmu_h\rangle_{L^2\Lambda^k}
\\
+\langle\odelta_k\fomega-\ftau_h,\odelta_{k,h}\fmu_h\rangle_{L^2\Lambda^{k-1}}-\langle\od^{k-1}(\odelta_k\fomega-\ftau_h),\fmu_h\rangle_{L^2\Lambda^k}
\\
\leqslant (\|\od^k\fomega-\feta_h\|_{L^2\Lambda^{k+1}}+\|\odelta_{k+1}(\od^k\fomega-\feta_h)\|_{L^2\Lambda^k})(\|\fmu_h\|_{L^2\Lambda^k}+\|\odelta_{k,h}\fmu_h\|_{L^2\Lambda^{k-1}})
\\
+(\|\odelta_k\fomega-\ftau_h\|_{L^2\Lambda^{k-1}}+\|\od^{k-1}(\odelta_k\fomega-\ftau_h))(\|\fmu_h\|_{L^2\Lambda^k}+\|\od^k_h\fmu_h\|_{L^2\Lambda^{k+1}}).
\end{multline}
This estimate, together with the approximation estimate Lemma \ref{lem:globalinterror}, leads to the total error estimation, and completes the proof. 
\end{proof}

\subsection{Implementation of the scheme: a set of locally supported basis functions}

The finite element space $\fV_{\od\cap\odelta}^{\rm m, +\od}\Lambda^k$ does not correspond to a ``finite element" defined as Ciarlet's triple \cite{Ciarlet.P1978book}. Though, in this section, we present a set of basis functions of $\fV_{\od\cap\odelta}^{\rm m, +\od}\Lambda^k$ which are tightly supported. Therefore, the finite element scheme can still be implemented by the standard routine. 

To illustrate the general procedure in Section \ref{sec:generalproce}, a two-dimensional example is given in Section \ref{sec:examples}, where we particularly refer to Figures \ref{fig:basisinteriorvertex} and \ref{fig:boundaryvertex} for the illustration of the local supports of the basis functions.

\subsubsection{A general procedure}\label{sec:generalproce}
Denote
$$
\pddedmpdL^k(T):=\left\{\langle\odelta_k \fmu,\ftau\rangle_{L^2\Lambda^{k-1}(T)}-\langle \fmu, \od^{k-1} \ftau\rangle_{L^2\Lambda^k(T)}=0,\ \forall\,\ftau\in \mathcal{P}^-_1\Lambda^{k-1}(T)\right\},
$$
and
$$
\pddedempdL^k(T):=\left\{\langle\od^k \fmu,\feta\rangle_{L^2\Lambda^{k+1}(T)}-\langle \fmu, \odelta_{k+1} \feta\rangle_{L^2\Lambda^k(T)}=0, \forall\,\feta\in\mathcal{P}^{*,-}_1\Lambda^{k+1}(T)\right\}.
$$
Then $\pddedmpdL^k(T)$ is unisolvent by $\left\{\langle\od^k \fmu,\feta\rangle_{L^2\Lambda^{k+1}(T)}-\langle \fmu, \odelta_{k+1} \feta\rangle_{L^2\Lambda^k(T)},\ \ \feta\in\mathcal{P}^{*,-}_1\Lambda^{k+1}(T)\right\}$, $\pddedempdL^k(T)$ is unisolvent by $\left\{\langle\odelta_k \fmu,\ftau\rangle_{L^2\Lambda^{k-1}(T)}-\langle \fmu, \od^{k-1} \ftau\rangle_{L^2\Lambda^k(T)},\ \ftau\in \mathcal{P}^-_1\Lambda^{k-1}(T)\right\}$. Denote 
$$
\pddedmpdL^k(\mathcal{G}_h)=\prod_{T\in\mathcal{G}_h}\pddedmpdL^k(T),\ \  \pddedempdL^k(\mathcal{G}_h)=\prod_{T\in\mathcal{G}_h}\pddedempdL^k(T).
$$ 
Then 
$$
\pddempdL^k(T)=\pddedmpdL^k(T)\oplus \pddedempdL^k(T)\ \ \mbox{and}\ \ \pddempdL^k(\mathcal{G}_h)=\pddedmpdL^k(\mathcal{G}_h)\oplus \pddedempdL^k(\mathcal{G}_h).
$$
It follows that
\begin{multline*}
\fV_{\od\cap\odelta}^{\rm m, +\od}\Lambda^k=\Big\{\fmu_h\in\pddempdL^k(\mathcal{G}_h):
\langle\od^k_h\fmu_h,\feta_h\rangle_{L^2\Lambda^{k+1}}-\langle\fmu_h,\odelta_{k+1}\feta_h\rangle_{L^2\Lambda^k}=0,\ \forall\,\feta_h\in\fW^*_{h0}\Lambda^{k+1},\quad
\\
\mbox{and}\ \ \langle\odelta_{k,h}\fmu_h,\ftau_h\rangle_{L^2\Lambda^{k-1}}-\langle\fmu_h,\od^{k-1}\ftau_h\rangle_{L^2\Lambda^k}=0,\ \forall\,\ftau_h\in\fW_h\Lambda^{k-1}
\Big\}
\\
=\Big\{\fmu_h\in\pddedmpdL^k(\mathcal{G}_h)\oplus \pddedempdL^k(\mathcal{G}_h):
\langle\od^k_h\fmu_h,\feta_h\rangle_{L^2\Lambda^{k+1}}-\langle\fmu_h,\odelta_{k+1}\feta_h\rangle_{L^2\Lambda^k}=0,\ \forall\,\feta_h\in\fW^*_{h0}\Lambda^{k+1},
\\
\mbox{and}\ \ \langle\odelta_{k,h}\fmu_h,\ftau_h\rangle_{L^2\Lambda^{k-1}}-\langle\fmu_h,\od^{k-1}\ftau_h\rangle_{L^2\Lambda^k}=0,\ \forall\,\ftau_h\in\fW_h\Lambda^{k-1}
\Big\}
\\
=\Big\{\fmu_h\in \pddedempdL^k(\mathcal{G}_h):\langle\odelta_{k,h}\fmu_h,\ftau_h\rangle_{L^2\Lambda^{k-1}}-\langle\fmu_h,\od^{k-1}\ftau_h\rangle_{L^2\Lambda^k}=0,\ \forall\,\ftau_h\in\fW_h\Lambda^{k-1}
\Big\}\qquad\qquad\qquad
\\
\oplus \Big\{\fmu_h\in\pddedmpdL^k(\mathcal{G}_h):\ \langle\od^k_h\fmu_h,\feta_h\rangle_{L^2\Lambda^{k+1}}-\langle\fmu_h,\odelta_{k+1}\feta_h\rangle_{L^2\Lambda^k}=0,\ \forall\,\feta_h\in\fW^*_{h0}\Lambda^{k+1}
\Big\}:=\fV_{\odelta}\oplus\fV_{\od}.
\end{multline*}
Now we figure out the basis functions of $\fV_{\odelta}$ and $\fV_{\od}$, respectively. They two form the whole set of basis functions of $\fV_{\od\cap\odelta}^{\rm m, +\od}\Lambda^k$.

\paragraph{\bf Basis functions of $\fV_{\odelta}$} The basis functions of $\fV_{\odelta}$ are determined by a 3-step procedure. 
\begin{description}
\item[Step 1]
Let $\mathbf{B}_{\bf W}^{k-1}=\{\fpsi_j\}_{j=1}^{\dim(\fW_h\Lambda^{k-1})}$ be the set of nodal basis functions of $\fW_h\Lambda^{k-1}$; then on every simplex $T$, the restrictions $\fpsi_j|_T$ of those $\fpsi_j$ that are nonzero on $T$ are linearly independent, and such $\fpsi_j|_T$ expand $\mathcal{P}^-_1\Lambda^{k-1}(T)$;
\item[Step 2] for any $T\in\mathcal{G}_h$, set $I^T:=\left\{1\leqslant i\leqslant \dim(\fW_h\Lambda^{k-1}):\mathring{T}\cap {\rm supp}(\fpsi_i)\neq\emptyset\right\}$, and there exist a set of functions $\left\{\fmu^T_i:i\in I^T\right\}\subset \pddedempdL^k(T)$, such that $\langle\odelta_k\fmu^T_i,\fpsi_j|_T\rangle_{L^2\Lambda^{k-1}}-\langle\fmu^T_i,\od^{k-1}\fpsi_j|_T\rangle_{L^2\Lambda^k}=\delta_{ij}$. Then $\pddedempdL^k(T)={\rm span}\{\fmu^T_i:\ i\in I^T\}$, and all $\fmu_i^T$ are linearly independent. 
\item[Step 3] Now, by definition,
\begin{multline*}
\fV_{\odelta}=\left\{\fmu_h\in \pddedempdL^k(\mathcal{G}_h):\langle\odelta_{k,h}\fmu_h,\ftau_h\rangle_{L^2\Lambda^{k-1}}-\langle\fmu_h,\od^{k-1}\ftau_h\rangle_{L^2\Lambda^k}=0,\ \forall\,\ftau_h\in\fW_h\Lambda^{k-1}\right\}
\\
=\left\{\fmu_h\in \sum_{T\in\mathcal{G}_h}\sum_{i\in I^T}{\rm span}\{E_T^\Omega\fmu_i^T\}:\langle\odelta_{k,h}\fmu_h,\fpsi_j\rangle_{L^2\Lambda^{k-1}}-\langle\fmu_h,\od^{k-1}\fpsi_j\rangle_{L^2\Lambda^k}=0,\ \forall\,\fpsi_j\in \mathbf{B}_{\bf W}^{k-1}\right\}
\\
=\sum_{1\leqslant j\leqslant \dim(\fW_h\Lambda^{k-1})}\left\{\fmu_h\in\sum_{\mathring{T}\cap{\rm supp}(\fpsi_j)\neq\emptyset}{\rm span}\{E_T^\Omega\fmu_j^T\}:\langle\odelta_{k,h}\fmu_h,\fpsi_j\rangle_{L^2\Lambda^{k-1}}-\langle\fmu_h,\od^{k-1}\fpsi_j\rangle_{L^2\Lambda^k}=0\right\}.
\end{multline*}
\end{description}
Namely, a set of basis functions of $\fV_{\odelta}$ consists of, for $1\leqslant j\leqslant \dim(\fW_h\Lambda^{k-1})$, functions 
\begin{equation}
\fmu_h\in\sum_{\mathring{T}\cap{\rm supp}(\fpsi_j)\neq\emptyset}{\rm span}\{E_T^\Omega\fmu_j^T\},\ \ \mbox{such\ that}\  \langle\odelta_{k,h}\fmu_h,\fpsi_j\rangle_{L^2\Lambda^{k-1}}-\langle\fmu_h,\od^{k-1}\fpsi_j\rangle_{L^2\Lambda^k}=0.
\end{equation}
Note that the supports of these functions are contained in the support of $\fpsi_j$. 

\paragraph{\bf Basis functions of $\fV_{\od}$} The basis functions of $\fV_{\od}$ are determined by a 3-step procedure, slightly different from the procedure for $\fV_{\odelta}$. They basis functions have basically two categories. 
\begin{description}
\item[Step 1]
Let $\mathbf{B}_{\bf W,0}^{*,k+1}=\{\fpsi^*_j\}_{j=1}^{\dim(\fW^*_{h0}\Lambda^{k+1})}$ be the set of nodal basis functions of $\fW^*_{h0}\Lambda^{k+1}$; then on every simplex $T$, the restrictions $\fpsi^*_j|_T$ of those $\fpsi^*_j$ that are nonzero on $T$ are linearly independent, and for interior $T$, such $\fpsi^*_j|_T$ expand $\mathcal{P}^{*,-}_1\Lambda^{k+1}(T)$;
\item[Step 2] For any $T\in\mathcal{G}_h$, set $I^T:=\left\{1\leqslant i\leqslant \dim(\fW^*_h\Lambda^{k+1}):\mathring{T}\cap {\rm supp}(\fpsi^*_i)\neq\emptyset\right\}$, and there exist, not necessarily unique, a set of functions $\left\{\fmu^T_i:i\in I^T\right\}\subset \pddedmpdL^k(T)$, such that $\langle\od^k\fmu^T_i,\fpsi^*_j|_T\rangle_{L^2\Lambda^{k+1}}-\langle\fmu^T_i,\odelta_{k+1}\fpsi^*_j|_T\rangle_{L^2\Lambda^k}=\delta_{ij}$. Denote $({\rm span}\{\fmu^T_i:\ i\in I^T\})^{\rm c}=\{\fmu\in \pddedmpdL^k(T): \langle\od^k\fmu,\fpsi^*_j|_T\rangle_{L^2\Lambda^{k+1}}-\langle\fmu,\odelta_{k+1}\fpsi^*_j|_T\rangle_{L^2\Lambda^k}=0,\ j\in I^T\}$, not necessarily non-empty. Then $$
\pddedmpdL^k(T)={\rm span}\{\fmu^T_i:\ i\in I^T\}+({\rm span}\{\fmu^T_i:\ i\in I^T\})^{\rm c}.
$$
\item[Step 3] By definition,
$$
\begin{array}{rl}
\fV_{\od}=&\displaystyle\Big\{\fmu_h\in\pddedmpdL^k(\mathcal{G}_h):\ \langle\od^k_h\fmu_h,\feta_h\rangle_{L^2\Lambda^{k+1}}-\langle\fmu_h,\odelta_{k+1}\feta_h\rangle_{L^2\Lambda^k}=0,\ \forall\,\feta_h\in\fW^*_{h0}\Lambda^{k+1}
\Big\}
\\
=&\displaystyle\prod_{T\in\mathcal{G}_h} ({\rm span}\{\fmu^T_i:\ i\in I^T\})^{\rm c}
\\
&\  \displaystyle\oplus\left\{\fmu_h\in \sum_{T\in\mathcal{G}_h}\sum_{i\in I^T}{\rm span}\{E_T^\Omega\fmu_i^T\}:\langle\od^k_h\fmu_h,\fpsi^*_h\rangle_{L^2\Lambda^{k+1}}-\langle\fmu_h,\odelta_{k+1}\fpsi^*_h\rangle_{L^2\Lambda^k}=0,\ \forall\,\fpsi^*_j\in \mathbf{B}_{\bf W,0}^{k-1}\right\}
\\
=&\displaystyle\prod_{T\in\mathcal{G}_h} ({\rm span}\{\fmu^T_i:\ i\in I^T\})^{\rm c}
\\
&\ \displaystyle\oplus \sum_{1\leqslant j\leqslant \dim(\fW_{h0}\Lambda^{k+1})}\left\{\fmu_h\in\sum_{\mathring{T}\cap{\rm supp}(\fpsi_j)\neq\emptyset}{\rm span}\{E_T^\Omega\fmu_j^T\}:\langle\od^k_h\fmu_h,\fpsi^*_h\rangle_{L^2\Lambda^{k+1}}-\langle\fmu_h,\odelta_{k+1}\fpsi^*_h\rangle_{L^2\Lambda^k}=0\right\}.
\end{array}
$$
\end{description}
Namely, a set of basis functions of $\fV_{\od}$ consists of two parts: 
\begin{enumerate}
\item Part I: $\displaystyle\prod_{T\in\mathcal{G}_h} ({\rm span}\{\fmu^T_i:\ i\in I^T\})^{\rm c}$. The supports of these functions are each contained in a single simplex.
\item Part II: for $1\leqslant j\leqslant \dim(\fW^*_{h0}\Lambda^{k+1})$, functions 
$$
\displaystyle\fmu_h\in\sum_{\mathring{T}\cap{\rm supp}(\fpsi^*_j)\neq\emptyset}{\rm span}\{E_T^\Omega\fmu_j^T\}\ \ \mbox{such\ that}\ 
\langle\od^k_h\fmu_h,\fpsi^*_j\rangle_{L^2\Lambda^{k+1}}-\langle\fmu_h,\odelta_{k+1}\fpsi^*_j\rangle_{L^2\Lambda^k}=0.
$$
Note that the supports of these functions are contained in the support of $\fpsi^*_j$. 
\end{enumerate}

%

%
\subsubsection{Examples}\label{sec:examples}

We take the two-dimensional $H\Lambda^1\cap H^*_0\Lambda^1$ problem for example. Let $\Omega$ be a polygon. The problem reads: find $\fomega\in H(\rot,\Omega)\cap H_0(\dv,\Omega)$, such that 
\begin{equation}
(\rot\fomega,\rot\fmu)+(\dv\fomega,\dv\fmu)+(\fomega,\fmu)=(\ff,\fmu),\ \forall\,\fmu\in H(\rot,\Omega)\cap H_0(\dv,\Omega). 
\end{equation}
The corresponding spaces are $H^1(\Omega)=H({\rm grad},\Omega)$ for $H\Lambda^0$ and $H^1_0(\Omega)=H_0(\curl,\Omega)$ for $H^*_0\Lambda^2$, respectively. We use for the conforming spaces of Whitney forms the linear element space $V^1_h$ for $H({\rm grad},\Omega)$ and $V^1_{h0}=V^1_h\cap H^1_0(\Omega)$ for $H_0(\curl,\Omega)$.

Let $\mathcal{T}_h$ be a shape-regular triangular subdivision of $\Omega$ with mesh size $h$, such that $\overline\Omega=\cup_{T\in\mathcal{T}_h}\overline T$, and every boundary vertex is connected to at least one interior vertex. Denote by $\mathcal{E}_h$, $\mathcal{E}_h^i$, $\mathcal{E}_h^b$, $\mathcal{X}_h$, $\mathcal{X}_h^i$, $\mathcal{X}_h^b$ and $\mathcal{X}_h^c$ the set of edges, interior edges, boundary edges, vertices, interior vertices, boundary vertices and corners, respectively.

In the setting, the shape function space is
$$
\pddempdL^1(T)=\mathbf{P}^{\rm m+rot}_{\rot\cap\dv}(T)={\rm span}\left\{\left(\begin{array}{c}1\\ 0\end{array}\right),\ \left(\begin{array}{c}0\\ 1\end{array}\right),\ \left(\begin{array}{c}\tilde{x}\\ \tilde{y}\end{array}\right),\ \left(\begin{array}{c}\tilde{y}\\ -\tilde{x}\end{array}\right),\ \left(\begin{array}{c}\tilde{y}^2\\ 0\end{array}\right),\ \left(\begin{array}{c}0\\ \tilde{x}^2\end{array}\right)\right\}.
$$

\begin{figure}[htbp]
\begin{tikzpicture}[scale=1.1]

\path 	coordinate (a0b0) at (0,0)
coordinate (a0b1) at (0,1)
coordinate (a0b2) at (0,2)
coordinate (a0b3) at (0,3)
coordinate (a0b4) at (0,4)
coordinate (a1b0) at (1,0)
coordinate (a1b1) at (1,1)
coordinate (a1b2) at (1,2)
coordinate (a1b3) at (1,3)
coordinate (a1b4) at (1,4)
coordinate (a2b0) at (2,0)
coordinate (a2b1) at (2,1)
coordinate (a2b2) at (2,2)
coordinate (a2b3) at (2,3)
coordinate (a2b4) at (2,4)
coordinate (a3b0) at (3,0)
coordinate (a3b1) at (3,1)
coordinate (a3b2) at (3,2)
coordinate (a3b3) at (3,3)
coordinate (a3b4) at (3,4)
coordinate (a4b0) at (4,0)
coordinate (a4b1) at (4,1)
coordinate (a4b2) at (4,2)
coordinate (a4b3) at (4,3)
coordinate (a4b4) at (4,4);

\draw[line width=.4pt]  (a0b0) -- (a0b4) ;
\draw[line width=.4pt]  (a1b0) -- (a1b4) ;
\draw[line width=.4pt]  (a2b0) -- (a2b4) ;
\draw[line width=.4pt]  (a3b0) -- (a3b4) ;
\draw[line width=.4pt]  (a4b0) -- (a4b4) ;
\draw[line width=.4pt]  (a0b0) -- (a4b0) ;
\draw[line width=.4pt]  (a0b1) -- (a4b1) ;
\draw[line width=.4pt]  (a0b2) -- (a4b2) ;
\draw[line width=.4pt]  (a0b3) -- (a4b3) ;
\draw[line width=.4pt]  (a0b4) -- (a4b4) ;
\draw[line width=.4pt]  (a0b0) -- (a4b0) ;

\draw[line width=.4pt]  (a0b3) -- (a1b4) ;
\draw[line width=.4pt]  (a0b2) -- (a2b4) ;
\draw[line width=.4pt]  (a0b1) -- (a3b4) ;
\draw[line width=.4pt]  (a1b0) -- (a4b3) ;
\draw[line width=.4pt]  (a2b0) -- (a4b2) ;
\draw[line width=.4pt]  (a3b0) -- (a4b1) ;
\draw[line width=.4pt]  (a0b0) -- (a4b4) ;

\draw[fill] (a3b2) circle [radius=0.05];
\node[below right] at (a3b2) {$A$};
\draw[fill] (a2b1) circle [radius=0.05];
\node[below right] at (a2b1) {$B$};
\draw[fill] (a1b0) circle [radius=0.05];
\node[below] at (a1b0) {$C$};
 
\begin{scope}
\fill[pattern=horizontal lines] (a2b1)--(a3b1)--(a4b2)--(a4b3)--(a3b3)--(a2b2)--(a2b1);
\fill[pattern=vertical lines] (a1b0)--(a2b0)--(a3b1)--(a3b2)--(a2b2)--(a1b1)--(a1b0);
\fill[pattern=north west lines] (a0b0)--(a2b0)--(a2b1)--(a1b1)--(a0b0);
\end{scope}


\path 	coordinate (ca0b0) at (5,0)
coordinate (ca1b0) at (6,0)
coordinate (ca2b0) at (7,0)
coordinate (ca2b1) at (7,1)
coordinate (ca1b1) at (6,1);

\draw[line width=.4pt]  (ca2b1) -- (ca2b0) -- (ca1b0) -- (ca0b0) -- (ca1b1) -- (ca2b1) -- (ca1b0) -- (ca1b1);


\path 	coordinate (ba1b0) at (8,0)
coordinate (ba2b0) at (9,0)
coordinate (ba3b1) at (10,1)
coordinate (ba3b2) at (10,2)
coordinate (ba2b2) at (9,2)
coordinate (ba1b1) at (8,1)
coordinate (ba2b1) at (9,1);

\draw[line width=.4pt]  (ba1b1) -- (ba2b2);
\draw[line width=.4pt]  (ba1b0) -- (ba3b2);
\draw[line width=.4pt]  (ba2b0) -- (ba3b1);
\draw[line width=.4pt]  (ba2b2) -- (ba3b2);
\draw[line width=.4pt]  (ba1b1) -- (ba3b1);
\draw[line width=.4pt]  (ba1b0) -- (ba2b0);
\draw[line width=.4pt]  (ba1b0) -- (ba1b1);
\draw[line width=.4pt]  (ba2b0) -- (ba2b2);
\draw[line width=.4pt]  (ba3b1) -- (ba3b2);


\path 	coordinate (aa2b1) at (11,1)
coordinate (aa3b1) at (12,1)
coordinate (aa4b2) at (13,2)
coordinate (aa4b3) at (13,3)
coordinate (aa3b3) at (12,3)
coordinate (aa2b2) at (11,2)
coordinate (aa3b2) at (12,2);

\draw[line width=.4pt]  (aa2b2) -- (aa3b3);
\draw[line width=.4pt]  (aa2b1) -- (aa4b3);
\draw[line width=.4pt]  (aa3b1) -- (aa4b2);
\draw[line width=.4pt]  (aa3b3) -- (aa4b3);
\draw[line width=.4pt]  (aa2b2) -- (aa4b2);
\draw[line width=.4pt]  (aa2b1) -- (aa3b1);
\draw[line width=.4pt]  (aa2b1) -- (aa2b2);
\draw[line width=.4pt]  (aa3b1) -- (aa3b3);
\draw[line width=.4pt]  (aa4b2) -- (aa4b3);

\begin{scope}
\fill[pattern=horizontal lines] (aa2b1)--(aa3b1)--(aa4b2)--(aa4b3)--(aa3b3)--(aa2b2)--(aa2b1);
\fill[pattern=vertical lines] (ba1b0)--(ba2b0)--(ba3b1)--(ba3b2)--(ba2b2)--(ba1b1)--(ba1b0);
\fill[pattern=north west lines] (ca0b0)--(ca2b0)--(ca2b1)--(ca1b1)--(ca0b0);
\end{scope}

\draw[fill] (aa3b2) circle [radius=0.05];
\node[below right] at (aa3b2) {$A$};
\draw[fill] (ba2b1) circle [radius=0.05];
\node[below right] at (ba2b1) {$B$};
\draw[fill] (ca1b0) circle [radius=0.05];
\node[below] at (ca1b0) {$C$};

\end{tikzpicture}
\caption{Illustration of supports of $\psi_A$, $\psi_{\rm B}$ and $\psi_C$, as well as $\psi^*_A$ and $\psi^*_{\rm B}$. $A$ (as well as $B$) denotes an interior vertex, and $C$ denotes a boundary vertex.}\label{fig:vertices}
\end{figure}
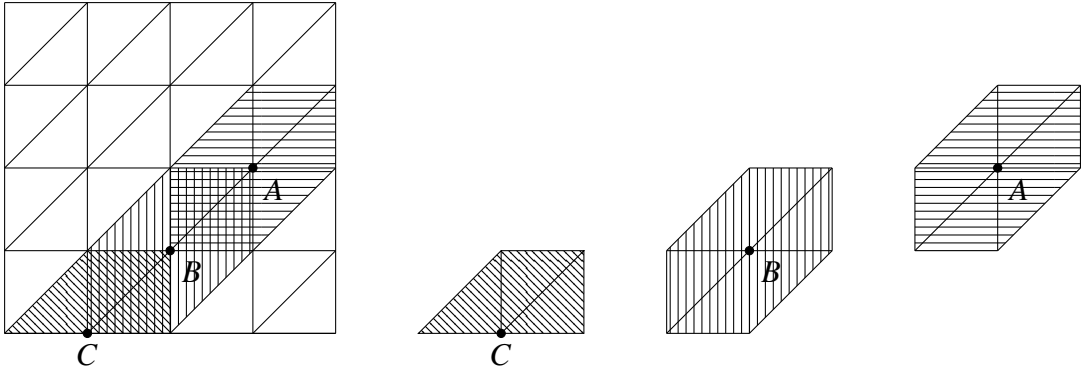

\begin{multline}\label{eq:defspaceintwod}
\fV_{\od\cap\odelta}^{\rm m, +\od}\Lambda^k=\fV_{\rot\cap\dv}^{\rm m, +\rot}:=\Bigg\{\fmu_h\in\prod_{T\in\mathcal{G}_h}\mathbf{P}^{\rm m+rot}_{\rot\cap\dv}(T):
(\rot_h\fmu_h,\feta_h)-(\fmu_h,\curl\feta_h)=0,\ \forall\,\feta_h\in V^1_{h0},
\\
\mbox{and}\ \ (\dv_h\fmu_h,\ftau_h)-(\fmu_h,{\rm grad}\ftau_h)=0,\forall\,\ftau_h\in V^1_h. \Bigg\}
\end{multline}

Now we present the basis functions of $\fV_{\od\cap\odelta}^{\rm m, +\od}\Lambda^k$ by a 3-step procedure. 

\paragraph{\bf Step 1} Both $V^1_h$ and $V^1_{h0}$ admit locally supported basis functions, denoted by $\phi_a$ and $\psi_a$ associated with vertices $a\in\mathcal{X}_h$ and $a\in\mathcal{X}_h^i$, respectively. The restrictions of $\phi_a$ and $\psi_a$ on a triangle are each one of the barycentric coordinates on the triangle. We refer to Figure \ref{fig:vertices} for an illustration of interior and boundary vertices, and also the supports of $\phi_a$($\psi_a$). 

\paragraph{\bf Step 2} On a cell $T$, with $a_i\in\mathcal{X}_h$ being its vertices, let $\lambda^{a_i}_T$ be the barycentric coordinates of $T$, and find $\fmu^{\dv}_{a_i,T},\ \fmu^{\rot}_{a_i,T}\in \pddempdL^1(T)$, $1\leqslant i\leqslant 3$, such that 
$$
(\rot \fmu^{\dv}_{a_i,T},\lambda_T^{a_j})_T-(\fmu^{\dv}_{a_i,T},\curl\lambda_T^{a_j})=0,\ \  (\dv\fmu^{\dv}_{a_i,T},\lambda_T^{a_j})_T-(\fmu^{\dv}_{a_i,T},{\rm grad}\lambda_T^{a_j})=\delta_{ij}, \ \ 1\leqslant i,j\leqslant 3,
$$
and
$$
(\rot \fmu^{\rot}_{a_i,T},\lambda_T^{a_j})_T-(\fmu^{\dv}_{a_i,T},\curl\lambda_T^{a_j})=\delta_{ij},\ \  (\dv\fmu^{\rot}_{a_i,T},\lambda_T^{a_j})_T-(\fmu^{\dv}_{a_i,T},{\rm grad}\lambda_T^{a_j})=0,\ \ 1\leqslant i,j\leqslant 3. 
$$

\paragraph{\bf Step 3} The global basis functions of $\fV_{\od\cap\odelta}^{\rm m, +\od}\Lambda^k$ fall into two categories, corresponding to $\fV_{\odelta}$ ($\fV_{\dv}$ here) and $\fV_{\od}$ ($\fV_{\rot}$ here), respectively. 
\begin{description}
\item[Category I] All these functions in ${\rm span}\{\fmu^{\dv}_{a,T}:a\in\mathcal{X}_h,\ T\in\mathcal{G}_h\}$ such that conditions in \eqref{eq:defspaceintwod} are satisfied with respect to every $\phi_a\in V^1_h$ a basis function. Particularly, with respect to any vertex $a$, the associated basis functions of $\fV_{\od\cap\odelta}^{\rm m, +\od}\Lambda^k$ are all these functions in ${\rm span}\{\fmu^{\dv}_{a,T}:\mathring{T}\cap {\rm supp}(\phi_a)\neq\emptyset\}$, namely, functions 
$$
\fomega_h=\sum_{\partial T\ni a}c_TE_T^\Omega\fmu^{\dv}_{a,T},\ \ \mbox{such\ that}\ 
\sum_{\partial T\ni a}(\dv\fomega_h,\phi_a|_T)_T+(\fomega_h,\nabla\phi_a|_T)_T=0.
$$
Those $\fomega_h$ in this category are each supported in a two-successive-cell patch. We refer to Figure \ref{fig:basisinteriorvertex} for the case $a\in\mathcal{X}_h^i$, and to Figure \ref{fig:boundaryvertex} (right) for an illustration that $a\in\mathcal{X}_h^b$. 

\item[Category II] All these functions in ${\rm span}\{\fmu^{\rot}_{a,T}:a\in\mathcal{X}_h,\ T\in\mathcal{G}_h\}$ such that conditions in \eqref{eq:defspaceintwod} are satisfied with respect to every $\psi_a\in V^1_{h0}$ a basis function. Particularly, with respect to an interior vertex $a$, the associated basis functions of $\fV_{\od\cap\odelta}^{\rm m, +\od}\Lambda^k$ are all these functions in ${\rm span}\{\fmu^{\rot}_{a,T}:\mathring{T}\cap {\rm supp}(\phi_a)\neq\emptyset\}$, namely, functions 
$$
\fomega_h=\sum_{\partial T\ni a}c_TE_T^\Omega\fmu^{\rot}_{a,T},\ \mbox{such\ that}\ \sum_{\partial T\ni a}(\rot\fomega_h,\phi_a|_T)_T-(\fomega_h,\curl\phi_a|_T)_T=0.
$$
Those $\fomega_h$ in this category are each supported in a two-successive-cell patch. We refer to Figure \ref{fig:basisinteriorvertex} for the case $a\in\mathcal{X}_h^i$. 

With respect to a boundary vertex $a$, the associated basis functions of $\fV_{\od\cap\odelta}^{\rm m, +\od}\Lambda^k$ are all these functions in ${\rm span}\{\fmu^{\rot}_{a,T}:\mathring{T}\cap {\rm supp}(\phi_a)\neq\emptyset\}$. The supports of the functions are each a single triangle. We refer to Figure \ref{fig:boundaryvertex}, left, for an illustration. 
\end{description}

\begin{figure}
\begin{tikzpicture}[scale=1.3]
\path 	
coordinate (a1mb0) at (0.9,0)
coordinate (a1pb0m) at (1.1,-0.1)
coordinate (a1pb0p) at (1.1,0.1)
coordinate (a1mb0m) at (0.9,-0.1)
coordinate (a1mb0p) at (0.9,0.1)
coordinate (a2pb0p) at (2.1,0.1)
coordinate (a2pb0m) at (2.1,-0.1)
coordinate (a2mb0p) at (1.9,0.1)
coordinate (a2mb0m) at (1.9,-0.1)
coordinate (a1b1m) at (1,0.9)
coordinate (a1mb1) at (0.9,1)
coordinate (a1pb1p) at (1.1,1.1)
coordinate (a1pb1m) at (1.1,0.9)
coordinate (a1mb1p) at (0.9,1.1)
coordinate (a1mb1m) at (0.9,0.9)
coordinate (a2b1m) at (2,0.9)
coordinate (a2pb1p) at (2.1,1.1)
coordinate (a2pb1m) at (2.1,0.9)
coordinate (a2mb1p) at (1.9,1.1)
coordinate (a2mb1m) at (1.9,0.9)
coordinate (a3pb1p) at (3.1,1.1)
coordinate (a3pb1m) at (3.1,0.9)
coordinate (a3mb1p) at (2.9,1.1)
coordinate (a3mb1m) at (2.9,0.9)
coordinate (a1mb2) at (0.9,2)
coordinate (a1pb2p) at (1.1,2.1)
coordinate (a1pb2m) at (1.1,1.9)
coordinate (a1mb2p) at (0.9,2.1)
coordinate (a1mb2m) at (0.9,1.9)
coordinate (a2pb2) at (2.1,2)
coordinate (a2b2p) at (2,2.1)
coordinate (a2b2m) at (2,1.9)
coordinate (a2mb2) at (1.9,2)
coordinate (a2pb2p) at (2.1,2.1)
coordinate (a2pb2m) at (2.1,1.9)
coordinate (a2mb2p) at (1.9,2.1)
coordinate (a2mb2m) at (1.9,1.9)
coordinate (a3pb2) at (3.1,2)
coordinate (a3pb2p) at (3.1,2.1)
coordinate (a3pb2m) at (3.1,1.9)
coordinate (a3mb2p) at (2.9,2.1)
coordinate (a3mb2m) at (2.9,1.9)
coordinate (a2b3p) at (2,3.1)
coordinate (a2pb3p) at (2.1,3.1)
coordinate (a2pb3m) at (2.1,2.9)
coordinate (a2mb3p) at (1.9,3.1)
coordinate (a2mb3m) at (1.9,2.9)
coordinate (a3pb3) at (3.1,3)
coordinate (a3b3p) at (3,3.1)
coordinate (a3pb3p) at (3.1,3.1)
coordinate (a3pb3m) at (3.1,2.9)
coordinate (a3mb3p) at (2.9,3.1)
coordinate (a3mb3m) at (2.9,2.9);

\draw[line width=.4pt]  (a1mb1) -- (a1mb2) -- (a2mb2) -- cycle;
\draw[line width=.4pt]  (a1b1m) -- (a2b2m) -- (a2b1m) -- cycle;
\draw[line width=.4pt]  (a1mb2p) -- (a2mb2p) -- (a2mb3p) -- cycle;
\draw[line width=.4pt]  (a2b2p) -- (a2b3p) -- (a3b3p) -- cycle;
\draw[line width=.4pt]  (a2pb2) -- (a3pb2) -- (a3pb3) -- cycle;
\draw[line width=.4pt]  (a2pb2m) -- (a3pb2m) -- (a2pb1m) -- cycle;

\node[left] at (2,2.6) {$+$};
\node[above] at (1.5,2) {$+$};
\node[left] at (1.4,2.6) {$-$};

\node at(1.5,2.4){$T_1$};

\node[right] at (1.9,2.7) {$+$};
\node[above] at (2.45,2.5) {$+$};
\node[above] at (2.5,3) {$-$};

\node at(2.3,2.9) {$T_6$};

\node[right] at (3,2.4) {$-$};
\node[above] at (2.7,1.9) {$+$};
\node[right] at (2.4,2.4) {$+$};

\node at(2.9,2.3) {$T_5$};

\node[right] at (2,1.35) {$+$};
\node[right] at (2.5,1.3) {$-$};
\node[below] at (2.6,2) {$+$};

\node at(2.4,1.5){$T_4$};

\node[right] at (1.3,1.3) {$+$};
\node[left] at (2.1,1.3) {$+$};
\node[below] at (1.5,1) {$-$};

\node at(1.7,1.1){$T_3$};

\node[left] at (1,1.5) {$-$};
\node[left] at (1.55,1.5) {$+$};
\node[below] at (1.3,2.1) {$+$};

\node at(1.1,1.7){$T_2$};

\draw[fill] (2,2) circle [radius=0.05];
\node[below right] at (2,2) {$A$};


\path 	coordinate (aa2b1) at (5,1)
coordinate (aa3b1) at (6,1)
coordinate (aa4b2) at (7,2)
coordinate (aa4b3) at (7,3)
coordinate (aa3b3) at (6,3)
coordinate (aa2b2) at (5,2)
coordinate (aa3b2) at (6,2);

\draw[line width=.4pt]  (aa2b2) -- (aa3b3);
\draw[line width=.4pt]  (aa2b1) -- (aa4b3);
\draw[line width=.4pt]  (aa3b1) -- (aa4b2);
\draw[line width=.4pt]  (aa3b3) -- (aa4b3);
\draw[line width=.4pt]  (aa2b2) -- (aa4b2);
\draw[line width=.4pt]  (aa2b1) -- (aa3b1);
\draw[line width=.4pt]  (aa2b1) -- (aa2b2);
\draw[line width=.4pt]  (aa3b1) -- (aa3b3);
\draw[line width=.4pt]  (aa4b2) -- (aa4b3);

\path 	coordinate (ba2b1) at (9,1)
coordinate (ba3b1) at (10,1)
coordinate (ba4b2) at (11,2)
coordinate (ba4b3) at (11,3)
coordinate (ba3b3) at (10,3)
coordinate (ba2b2) at (9,2)
coordinate (ba3b2) at (10,2);

\draw[line width=.4pt]  (ba2b2) -- (ba3b3);
\draw[line width=.4pt]  (ba2b1) -- (ba4b3);
\draw[line width=.4pt]  (ba3b1) -- (ba4b2);
\draw[line width=.4pt]  (ba3b3) -- (ba4b3);
\draw[line width=.4pt]  (ba2b2) -- (ba4b2);
\draw[line width=.4pt]  (ba2b1) -- (ba3b1);
\draw[line width=.4pt]  (ba2b1) -- (ba2b2);
\draw[line width=.4pt]  (ba3b1) -- (ba3b3);
\draw[line width=.4pt]  (ba4b2) -- (ba4b3);


\path 	coordinate (ca2b1) at (1,-2)
coordinate (ca3b1) at (2,-2)
coordinate (ca4b2) at (3,-1)
coordinate (ca4b3) at (3,0)
coordinate (ca3b3) at (2,0)
coordinate (ca2b2) at (1,-1)
coordinate (ca3b2) at (2,-1);

\draw[line width=.4pt]  (ca2b2) -- (ca3b3);
\draw[line width=.4pt]  (ca2b1) -- (ca4b3);
\draw[line width=.4pt]  (ca3b1) -- (ca4b2);
\draw[line width=.4pt]  (ca3b3) -- (ca4b3);
\draw[line width=.4pt]  (ca2b2) -- (ca4b2);
\draw[line width=.4pt]  (ca2b1) -- (ca3b1);
\draw[line width=.4pt]  (ca2b1) -- (ca2b2);
\draw[line width=.4pt]  (ca3b1) -- (ca3b3);
\draw[line width=.4pt]  (ca4b2) -- (ca4b3);

\path 	coordinate (da2b1) at (5,-2)
coordinate (da3b1) at (6,-2)
coordinate (da4b2) at (7,-1)
coordinate (da4b3) at (7,0)
coordinate (da3b3) at (6,0)
coordinate (da2b2) at (5,-1)
coordinate (da3b2) at (6,-1);

\draw[line width=.4pt]  (da2b2) -- (da3b3);
\draw[line width=.4pt]  (da2b1) -- (da4b3);
\draw[line width=.4pt]  (da3b1) -- (da4b2);
\draw[line width=.4pt]  (da3b3) -- (da4b3);
\draw[line width=.4pt]  (da2b2) -- (da4b2);
\draw[line width=.4pt]  (da2b1) -- (da3b1);
\draw[line width=.4pt]  (da2b1) -- (da2b2);
\draw[line width=.4pt]  (da3b1) -- (da3b3);
\draw[line width=.4pt]  (da4b2) -- (da4b3);

\path 	coordinate (ea2b1) at (9,-2)
coordinate (ea3b1) at (10,-2)
coordinate (ea4b2) at (11,-1)
coordinate (ea4b3) at (11,0)
coordinate (ea3b3) at (10,0)
coordinate (ea2b2) at (9,-1)
coordinate (ea3b2) at (10,-1);

\draw[line width=.4pt]  (ea2b2) -- (ea3b3);
\draw[line width=.4pt]  (ea2b1) -- (ea4b3);
\draw[line width=.4pt]  (ea3b1) -- (ea4b2);
\draw[line width=.4pt]  (ea3b3) -- (ea4b3);
\draw[line width=.4pt]  (ea2b2) -- (ea4b2);
\draw[line width=.4pt]  (ea2b1) -- (ea3b1);
\draw[line width=.4pt]  (ea2b1) -- (ea2b2);
\draw[line width=.4pt]  (ea3b1) -- (ea3b3);
\draw[line width=.4pt]  (ea4b2) -- (ea4b3);

\draw[fill] (aa3b2) circle [radius=0.05];
\node[below right] at (aa3b2) {$A$};
\draw[fill] (ba3b2) circle [radius=0.05];
\node[below right] at (ba3b2) {$A$};
\draw[fill] (ca3b2) circle [radius=0.05];
\node[above left] at (ca3b2) {$A$};
\draw[fill] (da3b2) circle [radius=0.05];
\node[above left] at (da3b2) {$A$};
\draw[fill] (ea3b2) circle [radius=0.05];
\node[below right] at (ea3b2) {$A$};

\begin{scope}
\fill[pattern=horizontal lines] (aa2b1)--(aa3b2)--(aa3b3)--(aa2b2)--(aa2b1);
\fill[pattern=horizontal lines] (ba2b1)--(ba3b1)--(ba3b2)--(ba2b2)--(ba2b1);
\fill[pattern=horizontal lines] (ca2b1)--(ca3b1)--(ca4b2)--(ca3b2)--(ca2b1);
\fill[pattern=horizontal lines] (da3b1)--(da4b2)--(da4b3)--(da3b2)--(da3b1);
\fill[pattern=horizontal lines] (ea3b2)--(ea4b2)--(ea4b3)--(ea3b3)--(ea3b2);
\end{scope}

\end{tikzpicture}

\caption{$A$ is an interior vertex; cf. Figure \ref{fig:vertices}. Five basis functions associated with the interior vertex $A$. The shadowed parts are respectively the supports of the basis functions.}\label{fig:basisinteriorvertex}
\end{figure}

\begin{figure}[htbp]
\begin{tikzpicture}[scale=1.5]

\path 	coordinate (ca0mb0) at (0.9,0)
coordinate (ca1mb0) at (1.9,0)
coordinate (ca1mb1) at (1.9,1)

coordinate (ca1b0p) at (2,0.1)
coordinate (ca1b1p) at (2,1.1)
coordinate (ca2b1p) at (3,1.1)

coordinate (ca1pb0) at (2.1,0)
coordinate (ca2pb0) at (3.1,0)
coordinate (ca2pb1) at (3.1,1);

\draw[line width=.4pt]  (ca0mb0) -- (ca1mb0) -- (ca1mb1) --cycle;
\draw[line width=.4pt]  (ca1b0p) -- (ca1b1p) -- (ca2b1p) --cycle;
\draw[line width=.4pt]  (ca1pb0) -- (ca2pb0) -- (ca2pb1) --cycle;

\node[left] at (2,0.5) {$+$};
\node[above] at (1.5,-0.1) {$+$};
\node[left] at (1.5,.55) {$-$};

\node[right] at (1.9,.7) {$+$};
\node[left] at (2.65,.7) {$+$};
\node[above] at (2.5,1) {$-$};

\node[right] at (2.5,.5) {$+$};
\node[above] at (2.65,-0.1) {$+$};
\node[right] at (3,.5) {$-$};

\draw[fill] (2,0) circle [radius=0.05];
\node[below] at (2,0) {$C$};

\node[below] at (1.8,-0.3) {(a) no essential boundary condition};


\path 	coordinate (ca0b0) at (5,0)
coordinate (ca1b0) at (6,0)
coordinate (ca2b0) at (7,0)
coordinate (ca2b1) at (7,1)
coordinate (ca1b1) at (6,1);

\draw[line width=.4pt]  (ca2b1) -- (ca2b0) -- (ca1b0) -- (ca0b0) -- (ca1b1) -- (ca2b1) -- (ca1b0) -- (ca1b1);

\path 	coordinate (cra0b0) at (8,0)
coordinate (cra1b0) at (9,0)
coordinate (cra2b0) at (10,0)
coordinate (cra2b1) at (10,1)
coordinate (cra1b1) at (9,1);

\draw[line width=.4pt]  (cra2b1) -- (cra2b0) -- (cra1b0) -- (cra0b0) -- (cra1b1) -- (cra2b1) -- (cra1b0) -- (cra1b1);

\begin{scope}
\fill[pattern=north west lines] (ca0b0)--(ca1b0)--(ca2b1)--(ca1b1)--(ca0b0);
\fill[pattern=north west lines] (cra1b0)--(cra2b0)--(cra2b1)--(cra1b1)--(cra1b0);
\end{scope}

\draw[fill] (ca1b0) circle [radius=0.05];
\node[below] at (ca1b0) {$C$};
\draw[fill] (cra1b0) circle [radius=0.05];
\node[below] at (cra1b0) {$C$};

\node[below] at (7.5,-0.3) {(b) with essential boundary condition };

\end{tikzpicture}

\caption{ C is boundary vertex with a three-cell patch; cf. Figure \ref{fig:vertices}. The basis functions associated with $C$ are:  (a) with no essential boundary condition, supported on a single cell; the case applies for $\fV_{\rot}$; (b) with essential boundary condition, supported on the shadowed parts; the case applies for $\fV_{\dv}$.}\label{fig:boundaryvertex}
\end{figure}

\appendix

\section{Some detailed calculation}

Let $\ixalpha\in \mathbb{IX}_{k,n}$ and $\ixbeta\in\mathbb{IX}_{n-k,n}$ be such that $\ixalpha$ and $\ixbeta$ partition $\{1, . . . , n\}$. Elementary calculation leads to that 
$$
\star(\dx^{\ixalpha_1}\wedge\dots \wedge \dx^{\ixalpha_k})=(-1)^{\rm pm}(\dx^{\ixbeta_1}\wedge\dots\wedge \dx^{\ixbeta_{n-k}}),\ \ {\rm pm}=\sum_{j=1}^k\ixalpha_j-\frac{k(k+1)}{2};
$$
$$
\star(\dx^{\ixbeta_1}\wedge\dots\wedge\dx^{\ixbeta_{n-k}})=(-1)^{\rm pm}(\dx^{\ixalpha_1}\wedge\dots \wedge \dx^{\ixalpha_k}),\ \ {\rm pm}=\sum_{j=1}^{n-k}\ixbeta_j-\frac{(n-k)(n-k+1)}{2};
$$
\begin{multline*}
\dx^{\ixalpha_i}\wedge\dx^{\ixbeta_1}\wedge\dots\wedge\dx^{\ixbeta_{n-k}}=(-1)^{\rm pm}\dx^{\ixbeta_1}\wedge\dots\wedge\dx^{\ixbeta_{m-1}}\wedge\dx^{\ixalpha_i}\wedge\dx^{\ixbeta_m}\wedge\dots\wedge\dx^{\ixbeta_{n-k}},
\\
\mbox{such\ that}\ \ixbeta_{m-1}<\ixalpha_i<\ixbeta_m,\ {\rm pm}=\ixalpha_i-i. 
\end{multline*}

$$
\begin{array}{l}
\displaystyle (-1)^{kn}\odelta_k\tilde\fmu^{\ixalpha}_{\odelta,T}=\star\od^{n-k}\star\fmu^{\ixalpha}_{\odelta,T}=\star\od^{n-k}(\sum_{j=1}^k(\tilde x^{\ixalpha_j})^2\star(\dx^{\ixalpha_1}\wedge\dots\wedge\dx^{\ixalpha_k}))
\\
\displaystyle =\star\od^{n-k}[\sum_{j=1}^k(\tilde x^{\ixalpha_j})^2(-1)^{\rm pm_1}(\dx^{\ixbeta_1}\wedge\dots\wedge\dx^{\ixbeta_{n-k}})]
=\star\sum_{j=1}^k2\tilde x^{\ixalpha_j}(-1)^{\rm pm_1}\dx^{\ixalpha_j}\wedge\dx^{\ixbeta_1}\wedge\dots\wedge \dx^{\ixbeta_{n-k}}
\\
\displaystyle =\star\sum_{j=1}^k 2\tilde x^{\ixalpha_j}(-1)^{\rm pm_1}(-1)^{\rm pm_2}\dx^{\ixbeta_1}\wedge\dots\wedge \dx^{\ixbeta_{m-1}}\wedge\dx^{\ixalpha_j}\wedge\dx^{\ixbeta_m}\wedge\dots\wedge\dx^{\ixbeta_{n-k}}
\\
\displaystyle =\sum_{j=1}^k(-1)^{\rm pm_1+pm_2+pm_3}2\tilde x^{\ixalpha_j}\dx^{\ixalpha_1}\wedge\dots\wedge\dx^{\ixalpha_{j-1}}\wedge\dx^{\ixalpha_{j+1}}\wedge\dx^{\ixalpha_k}
\\
\displaystyle =2\cdot (-1)^{(k-1)(n-1)}\sum_{j=1}^k(-1)^{j+1}\tilde x^{\ixalpha_j}\dx^{\ixalpha_1}\wedge\dots\wedge\dx^{\ixalpha_{j-1}}\wedge\dx^{\ixalpha_{j+1}}\wedge\dx^{\ixalpha_k}
\\
\displaystyle =2\cdot (-1)^{(k-1)(n-1)} \okappa_T(\dx^{\ixalpha_1}\wedge\dx^{\ixalpha_2}\wedge\dots\wedge \dx^{\ixalpha_k}),
\end{array}
$$
where 
${\rm pm_1}=\sum_{j=1}^k\ixalpha_j-\frac{k(k+1)}{2}$, ${\rm pm_2}=\ixalpha_j-j$, ${\rm pm_3}=\ixalpha_j+\sum_{l=1}^{n-k}\ixbeta_l-\frac{(n-k+1)(n-k+2)}{2}$, and ${\rm pm_1}+{\rm pm_2}+{\rm pm_3}=2\ixalpha_j-(k-1)(n-1)-(j+1).$
Namely 
\begin{equation}\label{eq:Aodelta}
\odelta_k\tilde\fmu^{\ixalpha}_{\odelta,T}=(-1)^n\cdot 2\okappa_T(\dx^{\ixalpha_1}\wedge\dx^{\ixalpha_2}\wedge\dots\wedge \dx^{\ixalpha_k}). 
\end{equation}

$$
\begin{array}{rl}
\od^k\tilde\fmu_{\od,T}^{\ixalpha}&\displaystyle=\sum_{j=1}^{n-k}2\tilde x^{\ixbeta_j}\dx^{\ixbeta_j}\wedge\dx^{\ixalpha_1}\wedge\dots\wedge\dx^{\ixalpha_k}
\\
&\displaystyle=\sum_{j=1}^{n-k}2\tilde x^{\ixbeta_j}(-1)^{\rm pm_1}\dx^{\ixalpha_1}\wedge\dots\wedge\dx^{\ixalpha_{m-1}}\wedge\dx^{\ixbeta_j}\wedge\dx^{\ixalpha_m}\wedge\dots\wedge\dx^{\ixalpha_k}\ \ \ (\ixalpha_{m-1}<\ixbeta_j<\ixalpha_m)
\\
&\displaystyle=\sum_{j=1}^{n-k} 2\tilde x^{\ixbeta_j}(-1)^{\rm pm_1}(-1)^{\rm pm_2}\star(\dx^{\ixbeta_1}\wedge\dots\wedge\dx^{\ixbeta_{j-1}}\wedge\dx^{\ixbeta_{j+1}}\wedge\dots\wedge\dx^{\ixbeta_{n-k}})
\\
&\displaystyle=2\cdot(-1)^{\rm pm_3} \star(\sum_{j=1}^{n-k}\tilde x^{\ixbeta_j}(-1)^{j+1}\dx^{\ixbeta_1}\wedge\dots\wedge\dx^{\ixbeta_{j-1}}\wedge\dx^{\ixbeta_{j+1}}\wedge\dots\wedge\dx^{\ixbeta_{n-k}})
\\
&\displaystyle=2\cdot(-1)^{\rm pm_3}  \star(\okappa_T(\dx^{\ixbeta_1}\wedge\dots\wedge \dx^{\ixbeta_{n-k}})) =2\cdot(-1)^{\rm pm_3} (-1)^{\rm pm_4} \star(\okappa_T(\star (\dx^{\ixalpha_1}\wedge\dots\dx^{\ixalpha_k})))
\\
&\displaystyle=-2(-1)^{(n-k)(k+1)}\star(\okappa_T(\star (\dx^{\ixalpha_1}\wedge\dots\dx^{\ixalpha_k}))),
\end{array}
$$
where ${\rm pm_1}=\ixbeta_j-j$, ${\rm pm_2}=\sum_{i=1}^{j-1}\ixbeta_i+\sum_{i=j+1}^{n-k}\ixbeta_i-\frac{1}{2}(n-k-1)(n-k)$, ${\rm pm_3}=\sum_{i=1}^{n-k}\ixbeta_i-\frac{1}{2}(n-k-1)(n-k)+1$, and ${\rm pm_4}=\sum_{i=1}^k\ixalpha_i-\frac{1}{2}k(k+1)$. 
Namely
\begin{equation}\label{eq:Aod}
\od^k\tilde\fmu_{\od,T}^{\ixalpha}=2(-1)^{n(1+k)+1}\star(\okappa_T(\star (\dx^{\ixalpha_1}\wedge\dots\dx^{\ixalpha_k}))).
\end{equation}

\end{document}